\documentclass[12pt]{amsart}
\usepackage{amssymb,latexsym,amsmath,tikz-cd}
\usepackage{enumitem}
\usepackage[mathscr]{eucal}
\usepackage{bbm,xcolor,comment}
\usepackage{mathdots}
\usepackage{chngcntr}

\usepackage{xr}
\numberwithin{equation}{section}
\numberwithin{table}{section} 

\newcommand{\hsp}[1]{{\hbox{\hspace{#1}}}}

\newcounter{letcnt1} 

\newcounter{letcnt2} 
\counterwithin{letcnt2}{letcnt1}

\def\b{\beta}
\def\d{\delta}
\def\e{\varepsilon}
\def\z{\zeta}
\def\s{\sigma}

\def\tAd{\mathrm{Ad}} 

\def\tAut{\mathrm{Aut}}
\def\sfa{\mathsf{a}}

\def\bC{\mathbb{C}} 
\def\cC{\mathcal{C}}

\def\tdet{\mathrm{det}}
\def\tdim{\mathrm{dim}}

\def\tEnd{\mathrm{End}}

\def\cF{\mathcal{F}} 

\def\sF{\mathscr{F}}

\def\bfG{\mathbf{G}}

\def\tGr{\mathrm{Gr}}
\def\fg{\mathfrak{g}}

\def\bi{\mathbf{i}}

\def\sfk{\mathsf{k}}

\def\fm{\mathfrak{m}}
\def\sfm{\mathsf{m}}

\def\sfn{\mathsf{n}}
\def\cO{\mathcal{O}}

\def\bQ{\mathbb{Q}} 

\def\bR{\mathbb{R}}

\def\tRe{\mathrm{Re}}

\def\sS{\mathscr{S}}
\def\tSL{\mathrm{SL}}

\def\tStab{\mathrm{Stab}}

\def\fsl{\mathfrak{sl}}

\def\tspan{\mathrm{span}}

\def\cU{\mathcal{U}}

\def\sfw{\mathsf{w}}

\def\sX{\mathscr{X}}

\def\bZ{\mathbb{Z}}

\def\half{\tfrac{1}{2}}

\def\tand{\quad\hbox{and}\quad}
\def\bs{\backslash}

\def\smallb{{\hbox{\small{$\bullet$}}}}
\def\inj{\hookrightarrow}
\def\sur{\twoheadrightarrow}

\def\op{\oplus}
\def\ot{\otimes}

\def\tw{\hbox{\small $\bigwedge$}}

\newtheorem{corollary}[equation]{Corollary}
\newtheorem{lemma}[equation]{Lemma}

\newtheorem{theorem}[equation]{Theorem}
\newtheorem{conjecture}[equation]{Conjecture}
\newtheorem*{mainthm*}{Main Theorem}
\newtheorem*{theorem*}{Theorem}
\newtheorem*{corollary*}{Corollary}
\newtheorem*{lemma*}{Lemma}

\theoremstyle{definition}

\newtheorem*{boldQ*}{Question}
\newtheorem*{boldP*}{Problem}

\theoremstyle{definition}

\theoremstyle{remark}

\newtheorem*{definition*}{Definition}
\newtheorem{example}[equation]{Example}
\newtheorem*{example*}{Example}
\newtheorem{remark}[equation]{Remark}

\newtheorem*{emphQ*}{Question}

\newenvironment{a.list}
  {\begin{enumerate}[label=\alph*.,itemsep=3pt,leftmargin=25pt,listparindent=20pt]}
  {\end{enumerate}}

\newenvironment{num.list}
  {
  \begin{enumerate}[itemsep=3pt,leftmargin=25pt,listparindent=20pt,label={\arabic*.}]
  }
  {\end{enumerate}}
\newenvironment{i_list}
  {\begin{enumerate}[label=(\roman*),itemsep=3pt,leftmargin=25pt,listparindent=20pt]}
  {\end{enumerate}}
\newenvironment{I_list}
  {\begin{enumerate}[label=(\Roman*),itemsep=3pt,leftmargin=25pt,listparindent=20pt]}
  {\end{enumerate}}
\newenvironment{i_list_emph}
  {\begin{enumerate}[label=\emph{(\roman*)},itemsep=3pt,leftmargin=25pt,listparindent=20pt]}
  {\end{enumerate}}

\def\bs{\backslash}
\def\ddb{\partial\overline{\partial}}

\def\sfd{\mathsf{d}}
\def\bfh{\mathbf{h}}
\def\sfh{\mathsf{h}}

\def\sfq{\mathsf{q}}
\def\fs{\mathfrak{s}}
\def\fz{\mathfrak{z}}

\def\olB{\overline B}

\def\GsD{\Gamma\bs D}
\def\tPhi{\tilde\Phi}
\def\PhiS{\Phi_\mathrm{SBB}}
\def\hPhiS{\Phi_\mathrm{SBB}'}

\def\hatP{\overline{\wp}{}'}
\def\olP{\overline{\wp}}

\begin{document}
\title{Extension of Hodge norms at infinity}



\author[Robles]{Colleen Robles}
\email{robles@math.duke.edu}
\address{Mathematics Department, Duke University, Box 90320, Durham, NC  27708-0320} 
\thanks{The research leading to these results received funding from the National Science Foundation (DMS 1906352), and from the European Union's Horizon 2020 research and innovation programme under the Marie Sk\l{}odowska-Curie grant agreement No 754340.}

\date{\today}

\begin{abstract}
It is a long-standing problem in Hodge theory to generalize the Satake--Baily--Borel compactification of a locally Hermitian symmetric space to arbitrary period maps.  A proper \emph{topological} Satake--Baily-Borel type completion has been constructed, and the problem of showing that the construction is \emph{algebraic} has been reduced to showing that the compact fibres $A$ of the completion admit neighborhoods $X$ satisfying certain properties.  All but one of those properties has been established; the outstanding problem is to show that holomorphic functions on certain divisors $Y \subset X$ ``at infinity'' extend to $X$.  Extension theorems of this type require that the complex manifold $X$ be pseudoconvex; that is, admit a plurisubharmonic exhaustion function.  The neighborhood $X$ is stratified, and the strata admit Hodge norms which are may be used to produce plurisubharmonic functions on the strata.  One would like to extend these norms to $X$ so that they may be used to construct  the desired plurisubharmonic exhaustion of $X$.  The purpose of this paper is show that there exists a function that \emph{simultaneously} extends all the Hodge norms along the strata that intersect the fibre $A$ nontrivially.
\end{abstract}
\maketitle

\section{Introduction}

Suppose that $D$ is a Mumford--Tate domain parameterizing pure, effective, weight $\sfn$, $Q$--polarized Hodge structures on a finite dimensional rational vector space $V$.  Fix a period map $\Phi : B \to \GsD$ defined on a smooth quasi-projective $B$ with smooth projective completion $\olB \supset B$ and simple normal crossing divisor $Z = \olB \bs B$ at infinity.  Let $\wp = \Phi(B) \subset \Gamma \bs D$ denote the image.  A proper \emph{topological} Satake--Baily--Borel (SBB) type completion $\PhiS : \olB \ \to \ \olP$ of $\Phi$ is constructed in \cite{GGR-part1}.  Let $Z_1,\ldots,Z_\nu$ denote the smooth irreducible components of $Z$, and $Z_I = \cap_{i\in I}\, Z_i$ the closed strata.  By the nilpotent orbit theorem \cite{MR0382272}, the period map $\Phi$ asymptotically induces a  period map $\Phi_I : Z_I^* \to \Gamma_I\bs D_I$ along the open strata $Z_I^* = Z_I \bs \cup_{j\not\in I}Z_j$.  Set $Z_\emptyset^* = B$, so that $\Phi_\emptyset = \Phi$.  The topological compactification $\olP$ of $\wp$ is the disjoint union of the images $\wp_I = \Phi_I(Z_I^*)$ modulo a certain equivalence relation that accounts for the fact that the period map $\Phi_I$ may extend to some $Z_J^* \subset Z_I$.  (When the extension exists, it coincides with $\Phi_J$.  See \cite[\S3]{GGR-part1} for details.)  With these identifications, $\wp_I \inj \olP$, the image $\olP$ is a finite union of quasi-projective varieties, and $\PhiS$ is continuous and proper.  The fibres of $\PhiS$ are algebraic subvarieties of $\olB$.  

One would like to assert that $\olP$ is itself projective algebraic.\footnote{After this work was completed, Bakker--Filipazzi--Mauri--Tsimerman exhibited the projective structure and extension $\Phi_\mathrm{SBB} : \olB \to \olP$, \cite{BFMT}.}  This is known to be the case when $D$ is hermitian and $\Gamma$ is arithmetic: $\olP$ is the closure of $\wp$ in the Satake--Baily--Borel compactification of $\Gamma\bs D$.  In general it is an open problem to show that $\olP$ is a complex analytic space.  The latter would imply Conjecture \ref{conj:ggr} below.

The completion $\PhiS$ admits a ``Stein factorization''
\[
\begin{tikzcd}
  \olB \arrow[r,"\hPhiS"]
  & \hatP \arrow[r]
  & \olP \,.
\end{tikzcd}
\]
The fibres of $\hatP \to \olP$ are finite, and the fibres of $\hPhiS$ are connected, compact algebraic subvarieties of $\olB$.  

\begin{conjecture}[{\cite{GGR-part1}}] \label{conj:ggr}
The topological space $\hatP$ is projective algebraic, and the map $\hPhiS : \olB \to \hatP$ is a morphism.  
\end{conjecture}

\noindent The conjecture holds in the case that $D$ is hermitian symmetric, and in the case that $\tdim\,\wp \le 2$, \cite{GGLR}.  The proof of Conjecture \ref{conj:ggr} has been reduced to showing that every (compact, connected) fibre $A$ of $\hPhiS$ admits a neighborhood $X \subset \olB$ with the following properties \cite[\S8]{GGR-part1}:
\begin{I_list}
\item \label{i:proper}
The restriction of $\hPhiS$ to $X$ is proper.
\item \label{i:extn}
Holomorphic functions on $Z_I \cap X$ extend to $X$.
\end{I_list}

\noindent Neighborhoods satisfying the first property \ref{i:proper} exist by Theorem \ref{T:ggr}.  Let $\cF^{\sfn} \subset \cF^{\sfn-1} \subset \cdots \subset \cF^0$ denote the Hodge vector bundles over $B$.  Assume that the local monodromy at infinity is unipotent.  Then the $\cF^p$ extend to $\olB$, \cite{MR1416353}.  Set $\sfk = \lceil (\sfn+1)/2 \rceil$.  Define
\begin{equation}\label{E:Lambda}
  \Lambda \ = \ \left\{
    \begin{array}{ll}
    \tdet(\cF^\sfn) \,=\, \tdet(\cF^1) \,,\quad & \sfn=1 \,,\\
    \tdet(\cF^\sfn) \ot \tdet(\cF^{\sfn-1}) \ot \cdots \ot 
    	\tdet(\cF^\sfk) \,,\quad & 
    \sfn  \hbox{ even,} \\
    \left[\tdet(\cF^\sfn) \ot \tdet(\cF^{\sfn-1}) 
    \ot \cdots \ot \tdet(\cF^{\sfk+1}) \right]^{\ot2} \, 
    	\ot \tdet(\cF^\sfk)\, \,, \  &
    \sfn \ge 3 \hbox{ odd.}
    \end{array}
  \right.
\end{equation}

\begin{theorem}[{\cite{GGR-part1}}] \label{T:ggr}
There exists a positive integer $\sfa \ge 1$ so that every fibre $A$ of $\hPhiS$ admits a neighborhood $X \subset \olB$ with the following properties: 
\begin{i_list_emph}
\item 
The restriction of $\hPhiS$ to $X$ is proper.
\item 
The line bundle $\Lambda^{\ot \sfa}$ is trivial over $X$.
\end{i_list_emph}
\end{theorem}

The second property \ref{i:extn} is an Ohsawa--Takegoshi type extension problem (although without the need for bounds on the $L^2$ norms) \cite{MR1782659, MR3525916}.  Such theorems usually impose the hypothesis that $X$ is pseudoconvex.

\subsection{Pseudoconvexity}

Recall that the neighborhood $X$ is \emph{pseudoconvex} if it admits a plurisubharmonic exhaustion function $\rho : X \to [-\infty,\infty)$.  A continuous function $\rho: X \to \bR$ is an \emph{exhaustion} if $\rho^{-1}[-\infty,r)$ is relatively compact for all $r \in \bR$.  The function is \emph{plurisubharmonic} (psh) if for every holomorphic map $\psi : \Delta \to X$, the composition $\rho \circ \psi$ is subharmonic.  If $\rho$ is $\cC^2$, then it is psh if and only if $\bi\ddb \rho \ge 0$.  For example, if $f \in \cO(X)$, then $\rho = |f|^2$ is psh.  Likewise, a line bundle with metric $h$ is positive if $-\log  h$ is psh.  Oka's Theorem asserts that a complex manifold is Stein if and only if it admits a smooth strictly psh exhaustion function.  

\begin{conjecture} \label{conj:pseudocnvx}
The neighborhood $X$ in Theorem \ref{T:ggr} may be chosen to be pseudoconvex.  There is a continuous exhaustion function $\rho : X \to [0,\infty)$ with the property that $\ddb \rho( v , \overline v) \ge 0$, and equality holds if and only if $v$ is tangent to a fibre of $\PhiS$.\footnote{This statement must be interpreted with some care as, in general, $\rho$ will be $\cC^1$, but not $\cC^2$.  Then the inequality $\ddb(v,\overline v) \rho \ge0$ of Conjecture \ref{conj:pseudocnvx} should be understood to allow $\ddb\rho(v,\overline v) =+\infty$.  The latter may arise when $v$ is normal $Z$, cf.~\cite{Robles-pseudocnvx-eg}.  In this case $\bi\ddb\rho$ is a positive current.  We expect that the restrictions $\bi\ddb\rho\big|_{Z_I^*}$  are both well-defined and smooth.}
\end{conjecture}

\noindent In \S\ref{S:triv} we will show that the conjecture holds in three cases: when $D$ is hermitian symmetric; when $A \subset B$; and when $A$ is a connected component of $Z$.   The purpose of this note is to discuss how the conjecture might be approached in general (\S\ref{S:prf-pcnvx}), and to establish a key result in that direction (extension of Hodge norms, Theorem \ref{T:h}).  This result is used elsewhere to prove the conjecture in the following simple, but nontrivial, case.

\begin{theorem}[{\cite{Robles-pseudocnvx-eg}}]\label{T:pseudocnvx}
Suppose that the Mumford--Tate domain $D$ parameterizes weight $\sfn=2$, effective, polarized Hodge structures with $p_g = h^{2,0} = 2$.  Assume that the fibre $A$ is contained in a codimension 1 strata $Z_i^*$.  Then Conjecture \ref{conj:pseudocnvx} holds.
\end{theorem}

\begin{remark}[Strict psh] 
The conjectural exhaustion function $\rho : X \to \bR$ will be the $\hPhiS$--pullback of a continuous function $\varrho$ on $\sX = \hPhiS(X) \subset \hatP$.  The assertion that $\ddb \rho( v , \overline v) \ge 0$, with equality precisely when $v$ is tangent to a fibre of $\PhiS$, should be interpreted as saying that $\varrho$ is a \emph{strictly} psh function on $\sX$.  This is ``interpretative'' because the topological space $\hatP$ is not yet shown to be complex analytic.  However, the space $\hatP$ is a finite union $\cup\,\hatP_\pi$ of complex analytic spaces, and the restriction $\varrho\big|_{\hatP_\pi}$ is strictly psh.
\end{remark}

\begin{remark}[Pseudoconvexity in Hodge theory]
Griffiths and Schmid showed that $D$ admits a smooth exhaustion function whose Levi form, restricted to the horizontal subbundle of the holomorphic tangent bundle, is positive definite at every point \cite[(8.1)]{MR0259958}.  In particular, the image of the lift $\tilde\Phi : \tilde B \to D$ to the universal cover of $B$ admits a strict psh exhaustion function.  
\end{remark}

\subsection{Discussion of Conjecture \ref{conj:pseudocnvx}} \label{S:triv}

Let $\Lambda \to \olB$ be the line bundled defined in \eqref{E:Lambda}.  Theorem \ref{T:ggr} implies $\Lambda^{\ot \sfa}$ is trivial over $X$ for some positive integer $\sfa\ge1$.

\subsubsection{Proof of Conjecture \ref{conj:pseudocnvx} when $D$ is hermitian}

It follows from \cite{MR0216035} and the triviality of $\Lambda^{\ot \sfa}\big|_X$ that there exist holomorphic functions $g_1 , \ldots, g_\mu : X \to \bC$ that separate the fibres of $\PhiS\big|_{X}$ and with the property that $V(g_1,\ldots,g_\mu) = A$.  Set $f = \sum |g_j|^2$.  Given a sufficiently small $\e>0$, and shrinking $X$ if necessary, we may assume that $X = \{ x \in X \ | \ f(x) < \e \}$.  Then $\rho = (\e - f)^{-1}$ is the desired psh function. 
\hfill\qed

\subsubsection{Proof of Conjecture \ref{conj:pseudocnvx} when $A \subset B$}

We may assume without loss of generality that $X \subset B$.  Then it follows from \cite[Theorem 6.14]{MR4557401} and the triviality of $\Lambda^{\ot \sfa}\big|_X$ that there exist holomorphic functions $g_1 , \ldots, g_\mu : X \to \bC$ that separate the fibres of $\PhiS\big|_{X}$.  Without loss of generality these functions have the property that $V(g_1,\ldots,g_\mu) = A$.  Now the argument above goes through verbatim.
\hfill\qed

\subsubsection{Proof of Conjecture \ref{conj:pseudocnvx} when $A$ is a connected component of $Z$} \label{S:pg}

Again it follows from \cite[Theorem 6.14]{MR4557401} and the triviality of $\Lambda^{\ot \sfa}\big|_X$ that there exist holomorphic functions $g_1 , \ldots, g_\mu : X \to \bC$ that separate the fibres of $\Phi\big|_{B\cap X}$ and with the property that $V(g_1,\ldots,g_\mu) = Z$.  \hfill\qed

\begin{proof}[A second proof in this case]
Let $h_0^\sfa$ be the Hodge norm-squared of a trivialization of $\Lambda^{\ot \sfa}\big|_{B\cap X}$.  Then $\rho = 1/h_0$ is a psh function, and the restriction $\rho \big|_{B \cap X}$ satisfies Conjecture \ref{conj:pseudocnvx}: we have $\ddb \rho( v , \overline v) \ge 0$, with equality if and only if $v$ is tangent to a fibre of $\Phi$, \cite{MR0259958}.  

The restriction $\rho\big|_{Z \cap X}$ vanishes identically.  So $\rho : X \to [0,\infty)$ is an exhaustion function if and only if $X \cap Z$ is compact.  Finally, $\rho\big|_{Z \cap X}$ will satisfy Conjecture \ref{conj:pseudocnvx} if and only if $A = X \cap Z$.
\end{proof}

\begin{proof}[A third proof in this case]
P.~Griffiths has pointed out that, if we allow $\rho$ to take value in $[-\infty,\infty)$, then $-\log h_0$ also yields a psh exhaustion with the  desired properties.
\end{proof}

\subsubsection{An approach to Conjecture \ref{conj:pseudocnvx} in the general case} \label{S:prf-pcnvx}

There are (at least) two possibilities for a continuous psh function $\rho_0 : X \to [0,\infty)$ with the property that the restriction $\rho_0 \big|_{B \cap X}$ satisfies Conjecture \ref{conj:pseudocnvx}: we have $\ddb \rho_0( v , \overline v) \ge 0$, with equality if and only if $v \in T(B\cap X)$ is tangent to a fibre of $\Phi$, cf.~\cite{Robles-pseudocnvx-eg}.  In both cases $\rho_0$ vanishes along $Z \cap X$, and so in general will not be an exhaustion function.  We need a second function $\rho_1 : X \to \bR$ with the following properties:
\begin{i_list}
\item
The restriction $\rho_1\big|_{Z \cap X}$ is psh.
\item \label{i:psh}
The sum $\rho_0 + \rho_1 : X \to \bR$ is psh.  In fact, $\ddb (\rho_0+\rho_1)( v , \overline v) \ge 0$, with equality precisely when $v$ is tangent to a fibre of $\PhiS$.
\item \label{i:A}
We have $\rho_0 + \rho_1 \ge 0$, and the fibre is characterized by 
\[
	A \ = \ \{ \rho_0 + \rho_1 \ = \ 0 \} \,.
\]
\end{i_list}
Then for sufficiently small $\e>0$ we may take $X = \{ \rho_0+\rho_1 < \e\}$ and $\rho = 1/(\e-\rho_0-\rho_1)$.

A natural source of psh functions on $Z_I^* \cap X$ are the $-\log h_I$ with $h_I^\sfa$ the Hodge norm-squared of trivialization  of $\Lambda^{\ot \sfa}\big|_X$.  The main result of this paper (Theorem \ref{T:h}) is the \emph{simultaneous} extension to $X$ of all the  $h_I$ with $Z_I^* \cap A$ non-empty.  And this extension does indeed yield a psh exhaustion of $X$, as outlined above, at least in two cases:  (i)  If $D$ is hermitian, then the extension is psh on $X$ (Theorem \ref{T:herm}).  (ii) The non-classical (non-hermitian) example of Theorem \ref{T:pseudocnvx}.  It is work in  progress to fully generalize these two examples.

\subsection{Extension of Hodge norms} \label{S:extnh-intro}

There is a Hodge metric associated to each $\Lambda^{\ot \sfa}\big|_{Z_I^*}$ that is canonically defined up to a positive multiple (\S\ref{S:hI}).  Fix a trivialization of $\Lambda^{\ot \sfa}\big|_X$ and let $h_I^\sfa : Z_I^* \cap X \to \bR_{>0}$ be the Hodge norm-squared of the trivialization.  Then $-\log h_I$ is a smooth psh function on $Z_I^* \cap X$.

\begin{theorem} \label{T:h}
The neighborhood $X$ of Theorem \ref{T:ggr} may be chosen so that it admits a continuous function $h : X \to \bR$ that is smooth on strata $Z_I^* \cap X$ \emph{(}including $B \cap X$\emph{)}, constant on $\hPhiS$--fibres, and has the following property: if $Z_I^* \cap A$ is nonempty, then the restriction of $h$ to $Z_I^*$ is a multiple of the $h_I$.  In particular the restriction of $-\log h$ to $Z_I^*$ is plurisubharmonic.
\end{theorem}

\noindent  The theorem is proved in \S\ref{S:prfTh}.

\begin{remark}
If the Mumford--Tate domain $D$ is hermitian, then $h$ is smooth and $-\log h$ is psh (Theorem \ref{T:herm}).  In general, $-\log h$ need not be psh \cite{Robles-pseudocnvx-eg}.  And smoothness is expected to fail when the hypotheses of Theorem \ref{T:sm} do not hold.
\end{remark}

\section{Preliminaries and review} \label{S:prelim}

The construction of $h$ will utilize the period matrix representation of an induced variation of Hodge structure; the latter is introduced in \S\ref{S:induced}, and the former is reviewed in \S\ref{S:pmr}.  The fibre $A$ is characterized in \S\ref{S:A}.

Set
\[
  G_\bR \,=\, \tAut(D) \,\subset\, \tAut(V_\bR,Q)
  \tand 
  G_\bC \,=\, \tAut(\check D) \,\subset\, \tAut(V_\bC,Q) \,.
\]

\subsection{Induced Hodge structure} \label{S:induced}

Let $\sfd_p = \tdim_\bC\,F^pV_\bC$ denote the dimensions of the Hodge filtration $F \in \check D$.  Any (pure, effective, polarized) Hodge structure $F \in D$ on $V$ naturally induces one on
\[
  H \ = \ 
  \left\{ \begin{array}{ll}
  \tw^{\sfd_\sfn}V \,=\, \tw^{\sfd_1}V \,,\quad &
  \sfn = 1 \,,\\
  (\tw^{\sfd_\sfn}V) \,\ot\,
  (\tw^{\sfd_{\sfn-1}} V) \,\ot \cdots \ot\,
  (\tw^{\sfd_\sfk} V) \,,\quad & \sfn \hbox{ even,}\\
  \left[(\tw^{\sfd_\sfn}V) \,\ot\,
  (\tw^{\sfd_{\sfn-1}} V) \,\ot \cdots \ot\,
  (\tw^{\sfd_{\sfk+1}} V)\right]^{\ot 2} 
  \ot (\tw^{\sfd_\sfk}V) \,,\quad & \sfn \ge 3 \hbox{ odd.}
  \end{array}\right.
\]
(Cf.~\cite[\S5]{GGR-part1}.)  We continue to denote the polarization by $Q$.  Let $\sfw$ denote the weight of the induced Hodge structure.  While the Hodge structure on $H$ is effective, if $\sfn\ge3$, then it will be the case that $h^{\sfw,0} = 0$.  Let $\sfq \ge0$ be the smallest integer such that $h^{\sfw-\sfq,\sfq}$ is nonzero.  Replacing $H$ with $H\ot\bQ(\sfq)$, and $\sfw$ with $\sfw-2\sfq$, we may assume that the pure, effective, polarized Hodge structure on $H$ satisfies $h^{\sfw,0}=1$.

From this point forward we will view the period map $\Phi$ as parameterizing pure, effective, weight $\sfw$, $Q$--polarized Hodge structures on $H$.

\begin{remark}[Relationship to $\Lambda$]
There is a tautological filtration  
\begin{equation}\label{E:F(H)}
  0 \,\not=\, \cF^\sfw(H_\bC) \,\subset\, \cF^{\sfw-1}(H_\bC) 
  \,\subset \cdots \subset\, 
  \cF^0(H_\bC) \ = \ \check D \times H_\bC
\end{equation}
of the trivial bundle.  The bundle $\cF^\sfw(H_\bC)$ has rank one, and the fibre over $F \in \check D$ is 
\[
  L \ = \ 
  \left\{ \begin{array}{ll}
  \tdet(F^\sfn V_\bC) \,=\, \tdet(F^1 V_\bC) \,,\quad & 
  \sfn = 1\,,\\
  \tdet(F^\sfn V_\bC) \,\ot\, \tdet(F^{\sfn-1} V_\bC) 
  \,\ot\cdots \ot\, \tdet(F^\sfk V_\bC) \,,\quad &
  \sfn \hbox{ even,} \\
  \left[\tdet(F^\sfn V_\bC) 
  \,\ot\cdots \ot\, \tdet(F^{\sfk+1} V_\bC)\right]^{\ot 2}
  \ot \tdet(F^\sfk V_\bC) \,,\quad &
  \sfn \ge 3 \hbox{ odd.}
  \end{array}\right.
\]
The bundles \eqref{E:F(H)} are $G_\bC$--homogeneous, and so all descend to $\Gamma \bs D$, and 
\[
  \Lambda\big|_B \ = \ \Phi^* \cF^\sfw(H_\bC) \,.
\]
\end{remark}

\subsection{Period matrix representation} \label{S:pmr}

The period map admits a period matrix representation over the neighborhood $X$ of Theorem \ref{T:ggr}.  This means we have the following structure.  (See \cite[\S4]{GGR-part1} for details.)   

\subsubsection{Monodromy about the fibre $A$}

Let $\pi_1(B \cap X) \sur \Gamma_X \subset \Gamma$ be the monodromy of the variation of Hodge structure over $B \cap X$.  Let $(W,F,\s)$ be any one of the limiting mixed Hodge structures arising along the fibre $A$.  The weight filtration $W$ is independent of this choice.  In general, both the Hodge filtration $F \in \check D$ and the nilpotent cone $\s$ of local monodromy logarithms depend on our choice, and are defined only up to the action of $\Gamma_X$.  (See Remark \ref{R:anotherchoice} for what we can say about other choices of $(W,F',\s')$ along $A$.)  Nonetheless, the limit
\[
  F_\infty \ = \ \lim_{y\to\infty} \exp(\bi y N) \cdot F
  \ \in \ \partial D
\]
is independent of our choice of both $(F,W,\s)$ and $N \in \s$.  (By convention the nilpotent cones here are nonzero and \emph{open}, so that $N \in \s$ is necessarily nonzero.)  Both $W$ and $F_\infty$ are invariant under the monodromy:
\begin{equation}\label{E:GammaX}
  \Gamma_X \ \subset \ \tStab_{G_\bR}(W) \,\cap\, 
  \tStab_{G_\bC}(F_\infty) \,, 
\end{equation}
where $G_\bR = \tAut(D)$ and $G_\bC = \tAut(\check D)$.

\subsubsection{Framing of the Hodge bundles} \label{S:framing}

Fix once and for all a limiting mixed Hodge structure $(W,F,\s)$ arising along the fibre $A$.  Choose a basis $\{e_0,e_1,\ldots,e_\sfd\}$ of $H_\bC$ so that $e_0$ spans $F^\sfw(H_\bC)$, $\{e_0,\ldots,e_{\sfd_{\sfw-1}}\}$ span $F^{\sfw-1}(H_\bC)$, and so on.  Since $F \in \check D$ is $Q$--isotropic (that is, satisfies the first Hodge--Riemann bilinear relation)
\[
  Q(F^p , F^q) \ = \ 0 \,,\qquad \forall \quad p+q < \sfw \,,
\]
we may assume that the basis also satisfies
\begin{equation}\label{E:Q}
	Q(e_i,e_j) \ = \ \d^\sfd_{i+j} \,,
	\qquad \forall \quad i \le \sfd/2 \,.
\end{equation}

Let $\cU \to B \cap X$ be the universal cover, and $\tPhi : \cU \to D$ the lift of the restricted period map $\Phi : B \cap X \to \Gamma_X \bs D$.  Then there exist holomorphic functions $z_j,z_{a,j},\ldots : \cU \to \bC$, so that 
\[
  \eta_0 \ = \ e_0 \ + \ \sum_{j=1}^\sfd z_j\,e_j
\]
frames $F^\sfw(\tPhi)$;  $\eta_0$ and 
\[
  \eta_a \ = \ e_a \ + \ \sum_{j=\sfd_{\sfw-1}+1}^\sfd z_{a,j}\,e_j 
  \,,\quad
  1 \le a \le \sfd_{\sfw-1}
\]
frame $F^{\sfw-1}(\tPhi)$; $\eta_0,\ldots\eta_{\sfd_{\sfw-1}}$ and 
\[
  \eta_a \ = \ e_a \ + \ \sum_{j=\sfd_{\sfw-2}+1}^\sfd z_{a,j}\,e_j 
  \,,\quad
  \sfd_{\sfw-1} + 1 \le a \le \sfd_{\sfw-2}
\]
frame $F^{\sfw-2}(\tPhi)$; and so on.  The framing $\{ \eta_0,\ldots , \eta_\sfd\}$ is the \emph{period matrix representation} of $\Phi$.  

We will sometimes treat the $z_j,z_{a,j},\ldots$ as holomorphic functions on $B \cap X \to \bC$ that are defined up to an action of the monodromy $\Gamma_X$.

\subsubsection{Schubert cell} \label{S:gpq}

Let $\sS \subset \check D$ be the open Schubert cell of filtrations $\tilde F\in \check D$ having generic intersection with $\overline{F_\infty}$.  The existence of a period matrix representation over $X$ is equivalent to the properties that (i) the action of $\Gamma_X$ on $\check D$ preserves $\sS$, and (ii) the lift $\tPhi$ of the period map takes value in $\sS \cap D$. Then the $z_j, z_{a,j} , \ldots$ are realized as the pullback of coordinates on $\sS$.

It will be convenient to have a description of this Schubert cell in terms of Deligne splittings, which we now review.  (See \cite{MR3701983} for details.)  The mixed Hodge structure $(W,F)$ determines a Deligne splitting
\[
  H_\bC \ = \ \bigoplus_{p,q\ge0} H^{p,q}_{W,F}
\]
that satisfies 
\begin{equation}\label{E:dsfilts}
  F^k \ = \ \bigoplus_{p\ge k}  H^{p,q}_{W,F} \,,\quad
  W_\ell \ = \ \bigoplus_{p+q\ge\ell}  H^{p,q}_{W,F} \tand
  F^k_\infty \ = \ \bigoplus_{q\le \sfw-k}  H^{p,q}_{W,F} \,,
\end{equation}
and
\[
  \overline{H^{p,q}_{W,F}} \ = \ H^{q,p}_{W,F}
  \quad\hbox{modulo}\quad 
  \bigoplus_{\substack{r<q \\ s<p}} H^{r,s}_{W,F} \,.
\]

Let $\fg_\bR$ and $\fg_\bC$ be the Lie algebras of $G_\bR$ and $G_\bC$ respectively.  The mixed Hodge structure $(W,F)$ on $H$ induces one on $\fg$.  Let 
\begin{equation}\label{E:gpq}
  \fg_\bC \ = \ \bigoplus_{p,q} \fg^{p,q}_{W,F}
\end{equation}
be the Deligne splitting.  The splitting satisfies
\begin{equation}\label{E:g(H)}
  \fg^{r,s}_{W,F}(H^{p,q}_{W,F}) \ \subset \ H^{p+r,q+s}_{W,F} \,,
\end{equation}
and respects the Lie bracket in the sense that 
\begin{equation}\label{E:adg}
  [ \fg^{p,q}_{W,F} \,,\, \fg^{r,s}_{W,F} ] 
  \ \subset \ \fg^{p+r,q+s}_{W,F} \,.
\end{equation}
The Lie algebra of $\tStab_{G_\bC}(F)$ is
\[
  \fs_F \ = \ \bigoplus_{p\ge0} \fg^{p,q}_{W,F} \,;
\]
the Lie algebra of $\tStab_{G_\bC}(W)$ is
\[
  \fs_W \ = \ \bigoplus_{p+q\le0}\fg^{p,q}_{W,F} \,;
\]
and the Lie algebra of $\tStab_{G_\bC}(F_\infty)$ is
\[
  \fs_\infty \ = \ \bigoplus_{q\le0} \fg^{p,q}_{W,F} \,.
\]
It follows from \eqref{E:GammaX} that the monodromy is contained in the Lie group
\begin{subequations}\label{SE:MX}
\begin{equation}\label{E:MX}
  \Gamma_X \ \subset \ M_X \ = \ \tStab_{G_\bC}(W) \,\cap\,
  \tStab_{G_\bC}(F_\infty) \,\cap\,\tStab_{G_\bC}(\overline{F_\infty})
\end{equation}
with Lie algebra 
\begin{equation}\label{E:mX}
   \fm_X \ = \ \bigoplus_{p,q\le0} \fg^{p,q}_{W,F} \,.
\end{equation}
\end{subequations}
Note that $M_X$ is defined over $\bR$.

The nilpotent algebra 
\begin{equation}\label{E:sFperp}
  \fs^\perp_F \ = \ \bigoplus_{p<0} \fg^{p,q}_{W,F}
\end{equation}
satisfies
\[
  \fg_\bC \ = \ \fs_F \ \op \ \fs_F^\perp \,.
\]
The exponential map $\fs^\perp_F \to \exp(\fs^\perp_F)$ is a biholomorphism.  The Schubert cell is the Zariski open
\[
    \sS \ = \ \exp(\fs^\perp_F) \cdot F \ \subset \ \check D \,;
\]
it is precisely the set of filtrations $\tilde F\in \check D$ having generic intersection with $\overline{F_\infty}$.  The maps $\fs^\perp_F \to \exp(\fs^\perp_F) \to \sS$ are biholomorphisms.  There is a holomorphic map 
\begin{equation}\label{E:eta}
	\eta : \cU \to \exp(\fs^\perp_F)
\end{equation}
such that 
\[
  \eta_j \ = \ \eta\cdot e_j \,;
\]
equivalently,
\[
  \tPhi \ = \ \eta \cdot F \,.
\]
Both $\eta$ and $\log \eta$ are equivalent to the framing $\{ \eta_j\}_{j=0}^\sfd$, and we will sometimes refer to both as the \emph{period matrix representation} of $\Phi$.

\begin{remark} \label{R:anotherchoice}
We have $\s \subset \fg^{-1,-1}_{W,F}$.  If $(W,F',\s')$ is any other limiting mixed Hodge structure arising along $A$, then 
\[
  F' \ \in \  \exp( \fm_X \cap \fs_F^\perp)
\]
and 
\[
  \s' \ \subset \ \bigoplus_{p,q\le-1} \fg^{p,q}_{W,F} \,.
\]
\end{remark}

\subsection{Definition of $e_\infty$ and $\sfm$} \label{S:einfty}

Since $\tdim\,F^\sfw(H_\bC) = 1$, there exists a unique $\sfw\le\sfm\le2\sfw$ so that
\[
  F^\sfw(H_\bC) \subset W_\sfm(H_\bC) 
  \tand 
  F^\sfw(H_\bC) \cap W_{\sfm-1}(H_\bC) = 0 \,.
\]
Symmetries in the limit mixed Hodge structure imply 
\begin{eqnarray*}
  \tdim\,W_{2\sfw-\sfm}(H_\bC) \cap F^{2\sfw-\sfm}(H_\bC) & = & 1\,,\\
  W_{2\sfw-\sfm-1}(H_\bC) \cap F^{2\sfw-\sfm}(H_\bC) & = & 0 \,,\\
  W_{2\sfw-\sfm}(H_\bC) \cap F^{2\sfw-\sfm+1}(H_\bC) & = & 0\,,
\end{eqnarray*}
and also that 
\[
  H_\bC \ = \ F^1(H_\bC) \ \op \ 
  \overline{W_{2\sfw-\sfm}(H_\bC) \cap F^{2\sfw-\sfm}(H_\bC)} \,.
\]
So we may assume that
\[
  e_\sfd \ \in \ 
  \overline{W_{2\sfw-\sfm}(H_\bC) \cap F^{2\sfw-\sfm}(H_\bC)} \,,
\] 
and $W_{2\sfw-\sfm}(H_\bC) \cap F^{2\sfw-\sfm}(H_\bC)$ is spanned by some $e_\infty \in \{e_0,e_1,\ldots,e_\sfd\}$.  (The reason for the subscript $\infty$ is discussed in Remark \ref{R:whyinfty}.)  To summarize, we have
\begin{eqnarray}
  \nonumber
  \tspan_\bC\,\{e_0\} & = & H^{\sfw,\sfm-\sfw}_{W,F} 
  \ = \ F^\sfw(H_\bC) \,,\\
  \label{E:H}
  \tspan_\bC\,\{e_\infty\} & = & H^{2\sfw-\sfm,0}_{W,F} 
  \ = \ W_{2\sfw-\sfm}(H_\bC) \cap F^{2\sfw-\sfm}(H_\bC)  \\
  \nonumber
  \tspan_\bC\,\{e_\sfd\} & = & H^{0,2\sfw-\sfm}_{W,F}(H_\bC)\,.
\end{eqnarray}

In order to get a feel for the vectors $e_0$ and $e_\infty$, and  the integer $\sfm$ it may be helpful to visualize them in the Hodge diamond for the mixed Hodge structure $(W,F)$ on $H$; see  \S\ref{S:hd} for two interesting examples.

\subsection{The fiber $A$} \label{S:A}

Without loss of generality, we may assume that there exist subsets $B_\ell \subset \{e_0 , e_1 , \ldots , e_\sfd\}$ so that $W_\ell(H_\bC)$ is framed by $\{ e_j \ |  \ j \in B_\ell \}$.  Let $\ell(j) = \min\{ \ell \ | \ e_j \in W_\ell(H_\bC)\}$.  Then the fibre $A \subset X$ is cut out by the equations
\begin{subequations}\label{SE:A}
\begin{equation}
  A \ = \ \{ \eta_j \ \equiv \ e_j
  \quad\hbox{mod}\quad W_{\ell(j)}(H_\bC) \ | \ 
  0 \le j \le \sfd\}\,.
\end{equation}
Equivalently, 
\begin{equation}
  A \ = \ \{ z_j  = 0 \ | \ \ell(j) \ge \sfm \}
  \,\cup\, \{ z_{a,j} = 0 \ | \ \ell(j) \ge \ell(a) \} \,.
\end{equation}
\end{subequations}
Although the $z_j$ and $z_{a,j}$ are defined only up to the action of monodromy as functions on $B \cap Z$, the fact that the monodromy preserves the weight filtration \eqref{E:GammaX} implies that the vanishing \eqref{SE:A} \emph{is} well-defined on $B \cap X$.

\section{Extension of Hodge norms} \label{S:prfTh}

The purpose of this section is to prove Theorem \ref{T:h}.  Take the neighborhood $X$ given by Theorem \ref{T:ggr}.  The problem is to construct a function $h : X \to \bR$ satisfying the properties of Theorem \ref{T:h}.  The function $h$ is defined over $B \cap X$ in \S\ref{S:dfnh}.  It is shown that this smooth function extends to a continuous function on $X$ in \S\ref{S:extnh}.  We will see that $h$ is constant on $\hPhiS$--fibres in \S\ref{S:cnst}.  Finally we show that the restriction of $h$ to $Z_I^*\cap X$ coincides with a positive multiple of $h_I$ whenever $Z_I^* \cap A$ is nonempty in \S\ref{S:hI}.

\subsection{Construction of $h$ over $B \cap X$} \label{S:dfnh}

Let $\eta_\infty = \eta \cdot e_\infty : \cU \to F^{2\sfw-\sfm}(\tPhi)$ denote the corresponding section (\S\S\ref{S:framing}--\ref{S:gpq}).  Define $0\not=\lambda \in \bC$ by 
\begin{equation}\label{E:lambda}
  \overline{\lambda\,e_\infty} \ = \ e_\sfd \,,
\end{equation}
and set
\begin{equation}\label{E:tildeh}
  \tilde h \ = \ \half\,Q(\eta_0 \,,\,\overline{\lambda\,\eta_\infty})
  \,+\, \half\,Q(\overline{\eta_0} \,,\,\lambda\,\eta_\infty) 
  \ = \ \tRe\,Q(\eta_0 \,,\,\overline{\lambda\,\eta_\infty}) \,.
\end{equation}
(It may be helpful to visualize the placement of $e_0,e_\infty,e_\sfd$ in the Hodge diamond of $(W,F)$.  See \S\ref{S:hd} for some examples.)

\begin{theorem} \label{T:descent}
The function $\tilde h : \cU \ \to \ \bR$ descends to $h:B \cap X \to \bR$.
\end{theorem}

\begin{proof}
To prove the theorem, we need to show that 
\begin{equation}\label{E:descent} 
  Q\left(\eta_0(u) , \overline{\eta_\infty(u)}\right) \ = \ Q\left(\eta_0(\gamma\cdot u) , \overline{\eta_\infty(\gamma \cdot u)}\right)
\end{equation}
for all $u \in \cU$ and $\gamma \in \pi_1(B \cap X)$, cf.~\eqref{E:tildeh}.  Note that $\pi_1(B \cap X)$ acts on both $\cU$ and $H_\bC$.  Given $u \in \cU$ and $\gamma \in \pi_1(B \cap X)$, we have
\[
  \gamma \cdot F^p(\tPhi(u)) \ = \ F^p(\tPhi(\gamma\cdot u)) \,.
\]
If $\eta_j(u) \in F^p(\tPhi(u))$, then it is likewise the case that both
\begin{equation}\label{E:gamma.eta}
  \gamma\cdot \eta_j(u)\,,\, \eta_j(\gamma \cdot u) 
  \ \in \ 
  F^p(\tPhi(\gamma\cdot u)) \,.
\end{equation}
However in general it need not be the case that $\gamma\cdot \eta_j(u) =  \eta_j(\gamma \cdot u)$, or even that they be linearly dependent.  Nonetheless, Lemma \ref{L:char} below does hold, and (keeping in mind that $\gamma \in \Gamma$ preserves the polarization $Q$) implies
\[
  Q\left( \eta_0(u) , \overline{\eta_\infty(u)} \right) 
  \ = \ 
  Q\left( \gamma \cdot \eta_0(u) , \gamma \cdot \overline{\eta_\infty(u)} \right) 
  \ = \ 
  Q\left( \eta_0(\gamma \cdot u) , \overline{\eta_\infty(\gamma \cdot u)} \right) \,,
\]
establishing the desired \eqref{E:descent}.
\end{proof}

\begin{lemma} \label{L:char}
There exists a character $\chi(\gamma) \in S^1 \subset \bC$ so that
\[
  \gamma \cdot \eta_0(u) \ = \ \chi(\gamma)\,\eta_0(\gamma \cdot u)
  \tand 
  \gamma \cdot \eta_\infty(u) \ = \ 
  \chi(\gamma)\,\eta_\infty(\gamma \cdot u) \,.
\] 
\end{lemma}

\begin{proof}
First we observe that \eqref{E:gamma.eta} and $\tdim\,F^\sfw(\tPhi(\gamma \cdot u)) = 1$ implies the first equality in the lemma: $\gamma \cdot \eta_0(u) = \chi(\gamma)\,\eta_0(\gamma \cdot u)$ for some character $\chi : \Gamma_X \to \bfG_{m,\bC}$.

Next we observe that \eqref{E:g(H)}, \eqref{E:MX} and 
\[
  e_\infty \ \in \ W_{2\sfw-\sfm}(H_\bC) \cap F^{2\sfw-\sfm}(H_\bC)
  \ = \ H^{2\sfw-\sfm,0}_{W,F}
\] 
imply that 
\begin{equation}\label{E:chiinfty}
  \gamma \cdot \eta_\infty(u) \ = \ 
  \chi_\infty(\gamma)\,\eta_\infty(\gamma \cdot u)
\end{equation}
for some character $\chi_\infty : \Gamma_X \to \bfG_{m,\bC}$.  

Let 
\[
  h^{p,q}_{W,F} \ = \ \tdim_\bC\, H^{p,q}_{W,F} \,.
\]
Then \eqref{E:H} yields
\begin{equation}\label{E:h}
\begin{array}{rcl}
  1 & = & h^{\sfw,\sfm-\sfw}_{W,F} \,,\
  h^{2\sfw-\sfm,0}_{W,F} \,,\\
  0 & = & h^{\sfw,q}_{W,F} \,,\ h^{p,0}_{W,F} \,,
  \quad \forall \ q\not=\sfm-\sfw \,,\ p \not=2\sfw-\sfm \,.
\end{array}
\end{equation}
We call the $(h^{p,q}_{W,F})$ the \emph{Hodge diamond of the mixed Hodge structure}.  They are conveniently visualized in the $(p,q)$--plane (\S\ref{S:hd}).  They satisfy the \emph{symmetries}
\begin{equation}\label{E:symmetries}
  h^{p,q}_{W,F} \ = \ h^{q,p}_{W,F}
  \tand
  h^{p,q}_{W,F} \ = \ h^{p-k,q-k}_{W,F} \,,
\end{equation}
where $k = p+q-\sfw$.  (The first equality holds for arbitrary mixed Hodge structures, the second holds for \emph{limiting/polarized} mixed Hodge structures.)  These symmetries imply
\[
  1 \ = \ h^{\sfm-\sfw,\sfw}_{W,F} \,,\ 
  h^{0,2\sfw-\sfm}_{W,F}  \,,
\]
and all other $h^{p,\sfw}_{W,F}$, $h^{0,q}_{W,F}$ are zero. 

The the desired \eqref{E:chiinfty} now follows from \eqref{E:g(H)}, \eqref{SE:MX}, \eqref{E:gamma.eta}, \eqref{E:H} and \eqref{E:h}.

It remains to show that $\chi(\gamma) = \chi_\infty(\gamma)$, and that this character has norm one.  Both are established in \cite[Lemma 5.10 and Theorem 5.21]{GGR-part1}: the $\chi(\gamma)$ and $\chi_\infty(\gamma)$ here are the $\chi(\b^{-1}) = \chi(\b)^{-1}$ and $\chi_\infty(\gamma)$ there, respectively; and the $\Gamma_X$ here is denoted $\Gamma_\infty$ there.
\end{proof}

\subsection{Continuous extension of $h$ to $X$} \label{S:extnh}

Since both $\eta_0$ and $\eta_\infty$ are holomorphic, the function $h$ is smooth on $B \cap X$.  

\begin{theorem} \label{T:extn}
The smooth function $h : B \cap X \to \bR$ extends to a continuous 
\[
  h : X \ \to \ \bR \,.
\]
The restriction to $Z_I^* \cap X$ is smooth.
\end{theorem}

\subsubsection{Outline of the proof of Theorem  \ref{T:extn}}

It it suffices to work locally: given a local coordinate chart $U \subset X$, we will show that the restriction $h\big|_{B \cap U}$ extends to a continuous function on $U$, that the extension is constant on $\hPhiS$--fibres, and restricts to a smooth function on the strata $Z_I^* \cap U$.  Sections \ref{S:loc1}--\ref{S:eta0} are occupied with studying the local coordinate expression of $Q(\eta_0,\overline{\eta_\infty})$ that is given by the nilpotent orbit theorem.  Following a comment on the limits to be analyzed in \S\ref{S:limcmt}, the meat of the argument is \S\ref{S:lim0} where properties of weight filtrations are used to establish the desired results.

\subsubsection{Local structure at infinity} \label{S:loc1}

Let $\Delta = \{ \tau \in \bC \ : \ |\tau| < 1 \}$ be the unit disc, and let $\Delta^* = \{ \tau \in \Delta \ : \ \tau\not=0\}$ be the punctured unit disc.  Set
\[
  \ell(\tau) \ = \ \frac{\log \tau}{2\pi\bi} \,.
\]

Fix a point $b\in X$.  There exists a local coordinate chart $t = (t_1,\ldots,t_r) : U \stackrel{\simeq}{\longrightarrow} \Delta^r$, centered at $b$, so that $B \cap U \simeq (\Delta^*)^k \times \Delta^{r-k}$.  Without loss of generality $U \subset X$.  By the nilpotent orbit theorem \cite{MR0382272} there exists a holomorphic function $\sF : U \to \check D$, and nilpotent operators $N_1 , \ldots , N_k \in \fg_\bQ$ so that the local coordinate representation of $\Phi$ is
\[
  \Phi(t) \ = \ 
  \exp \Big[ \sum_{j=1}^k \ell(t_j) N_j \Big] \cdot \sF(t) \,,
\]
modulo the local monodromy group $\Gamma_\mathrm{loc} \subset \Gamma_X$ generated by the $\exp(N_1),\ldots,\exp(N_k)$.  

Suppose that $b \in Z_J^* \cap X$.  Then $|J|=k$, and without loss of generality we may suppose that $J = \{1,\ldots,k\}$.  Given $I \subset  J$, we have
\[
  Z^*_I \cap U \ = \ 
  \{ t_i=0 \,, \ \forall \ i \in I \,;\ 
  	t_j \not=0 \,,\ \forall \ j \in J \bs I \} \,.
\]

Shrinking $X$ if necessary, there exists $K \supset  J$ so that $Z_K^* \cap A$ is nonempty.  This implies that the $N_1,\ldots,N_k$ generate a face of a nilpotent cone $\s'$ arising along $A$.  In particular, 
\begin{equation}\label{E:N}
  N_1,\ldots,N_k \ \in \ \bigoplus_{p,q\le-1}\fg^{p,q}_{W,F} 
  \ \subset \ \fs_F^\perp\,\cap\,W_{-2}(\fg_\bC) \,,
\end{equation}
by Remark \ref{R:anotherchoice}.  Since $\s' \subset \fs_F^\perp$, and $\Phi(t)$ takes value in $\Gamma_\mathrm{loc}\bs(\sS \cap D)$,  it follows that $\sF$ takes value in $\sS$.  So $\sF(t) = \zeta(t) \cdot F$ for  some holomorphic map $\zeta : U \to \exp(\fs_F^\perp)$.  In particular, the local coordinate representation of the function $\eta$ in \eqref{E:eta} is 
\begin{equation}\label{E:eta-loc0}
  \eta(t) \ = \ 
  \exp \Big[ \sum_{j=1}^k \ell(t_j) N_j \Big] \cdot \zeta(t) \,.
\end{equation}

\subsubsection{Local characterization of the fibre}\label{S:Aloc}

Suppose that $Z_J^* \cap A$ is nonempty.  Horizontality implies that the restriction of $\z$ to $Z_J^* \cap U = \{ t_1,\ldots,t_k=0 \}$ centralizes the $N_1,\ldots,N_k$, and therefore stabilizes the weight filtration $W$.  Equivalently, 
\begin{equation}\label{E:locW}
  \log \z \big|_{Z_J^* \cap U} \ \equiv \ 0
  \quad\hbox{modulo}\quad W_0(\fg_\bC) \,.
\end{equation}
It follows from the global characterization \eqref{SE:A} of the  fibre, that it is locally characterized by
\begin{equation}\label{E:locA}
  A \,\cap\, Z_J^* \,\cap\, U \ = \ 
  \left\{
  \log \z \big|_{Z_J^* \cap U} \ \equiv \ 0
  \quad\hbox{modulo}\quad W_{-1}(\fg_\bC) \right\} \,.
\end{equation}
Horizontality then implies that 
\begin{equation} \nonumber
  \log \z \big|_A \ \equiv \ 0
  \quad\hbox{modulo}\quad 
  \bigoplus_{p<-1,q \le 0}\fg^{p,q}_{W,F}
  \ \subset \ W_{-1}(\fg_\bC) \,.
\end{equation}

\subsubsection{A second local coordinate representation}\label{S:loc2}

It will be helpful to re-write \eqref{E:eta-loc0} as \eqref{E:eta-loc}.  Shrinking the coordinate neighborhood $U$ if necessary, we may assume that $Z_I^* \cap U$ is nonempty if and only if $I \subset J$.  Set 
\[
  \hat t_I \ = \ \prod_{j \not\in I} t_j \,.
\]
Then horizontality of the period map implies that $\zeta\big|_{Z_I^*\cap U}$ takes value in the centralizer $\fz_I$ of $\{ N_j \ | \ j \in I \}$.  So the holomorphic map $\log\z : U \to \fs_F^\perp$ may be expressed as 
\[
  \log \z \ = \ \sum_{I \subset J} \hat t_I \,f_I 
\]
with $f_I  : U  \to \fz_I \cap \fs_F^\perp$ holomorphic.  Set $f = f_\emptyset$.

Define 
\[
  \hat\theta_I(t) 
  \ = \ \exp \Big[ \sum_{j\not\in I} \ell(t_j) N_j  \Big] \,.
\]
Set
\begin{equation}\label{E:a}
   \hat\zeta(t) \ = \ 
   \exp \Big[ \sum_{I \subset J}
  	\hat t_I\,\tAd_{\hat\theta_I(t)}(f_I)
  \Big]
\end{equation}
and
\[
  \theta(t) \ = \ \hat\theta_\emptyset(t)
  \ = \ \exp \Big[ \sum_{j=1}^k \ell(t_j) N_j  \Big] \,.
\]
Then the local coordinate expression for $\eta$ in \eqref{E:eta-loc0} may be re-expressed as 
\begin{equation}\label{E:eta-loc}
  \eta(t) 
  \ = \ \theta(t) \cdot \zeta(t)
  \ = \ 
  \hat\z(t) \cdot \theta(t) \,.
\end{equation}

\begin{example} \label{eg:1}
If $J =  \{1,2\}$, then 
\[
  \log \z \ = \ f_{12} \,+\,
  t_1\,f_2 \,+\, t_2\,f_1 \,+\, t_1 t_2  \, f \,,
\]
with $f_{12}$ centralizing $N_1$ and $N_2$, and $f_j$ centralizing $N_j$.  We have 
\begin{eqnarray*}
  \sum_{I \subset J} \hat t_I\,\tAd_{\hat\theta_I(t)}(f_I)
  & = & f_{12} \,+\, t_1 \tAd_{\exp(\ell(t_1)N_1)}(f_2)
  \,+\, t_2 \tAd_{\exp(\ell(t_2)N_2)}(f_1) 
  \\ & & + \,
  t_1 t_2 \, \tAd_{\exp(\ell(t_1)N_1+\ell(t_2)N_2)}(f) \,.
\end{eqnarray*}
\end{example}

\subsubsection{Properties of $\hat\z(t)$}\label{S:a}

Since the $N_j$ are nilpotent elements of $\fg_\bC$, the $\tAd_{\hat\theta_I(t)}(f_I)$ are polynomial in the $\ell(t_j)$, $j \not\in I$.  This implies that the $\hat t_I \tAd_{\hat\theta_I(t)}(f_I)$ are continuous, albeit multivalued, functions on all of $U$ that are holomorphic on the strata $Z_I^* \cap U$.  By extension, $\hat\z(t)$ is continuous on all of $U$ and holomorphic on the strata $Z_I^* \cap U$.  We have limits
\begin{eqnarray*}
  \lim_{t_1,\ldots,t_k\to0} \
  \sum_{I \subset J} \hat t_I\,\tAd_{\hat\theta_I(t)}(f_I)
  & = & f_J \,,\\
  \lim_{\substack{t_i\to0 \\ i\in I' }} \
  \sum_{I \subset J} \hat t_I\,\tAd_{\hat\theta_I(t)}(f_I)
  & = & 
  \sum_{I' \subset I \subset J} \hat t_I\,\tAd_{I(t)}(f_I)
  \,.
\end{eqnarray*}
The first limit implies
\begin{equation}\label{E:lima}
  \lim_{t_1,\ldots,t_k\to0} 
  \hat\z(t) \ = \ \exp(f_J)
\end{equation}

\subsubsection{Continuity of $\eta_\infty$} \label{S:eta-infty}

It follows from Remark \ref{R:anotherchoice}, \eqref{E:g(H)} and \eqref{E:H} that $N_j(e_\infty) = 0$.  So
\[
  \theta(t) \cdot e_\infty \ = \ e_\infty \,,
\]
and
\begin{equation}\label{E:eta-infty}
  \eta_\infty(t) \ = \ \eta(t) \cdot e_\infty
  \ = \ \hat\z(t) \cdot e_\infty 
\end{equation}
is continuous (albeit defined only up to $\Gamma_\mathrm{loc}$).  In particular, the limit 
\begin{equation}\label{E:liminfty}
  \lim_{t_1,\ldots,t_k\to0} \eta_\infty
  \ = \ \exp(f_J) \cdot e_\infty
\end{equation}
exists, and is well-defined (independent of $\Gamma_\mathrm{loc}$).

\begin{remark} \label{R:whyinfty}
In the case that $b \in A$, the limit \eqref{E:liminfty} is a nonzero element of  
\[
  L_\infty \ = \ 
  \left\{ \begin{array}{ll}
  \tdet(F_\infty^\sfn V_\bC) \,=\, \tdet(F_\infty^1 V_\bC) 
  \,,\quad & \sfn = 1\,,\\
  \tdet(F_\infty^\sfn V_\bC) \,\ot\, 
  \tdet(F_\infty^{\sfn-1} V_\bC) \,\ot\cdots \ot\, 
  \tdet(F_\infty^\sfk V_\bC) \,,\quad & \sfn \hbox{ even,} \\
  \left[\tdet(F_\infty^\sfn V_\bC) \,\ot\cdots \ot\, 
  \tdet(F_\infty^{\sfk+1} V_\bC)\right]^{\ot 2}
  \ot \tdet(F_\infty^\sfk V_\bC) \,,\quad &
  \sfn \ge 3 \hbox{ odd.}
  \end{array}\right.
\]
\end{remark}

\subsubsection{Local coordinate representations of $\eta_0$ and  $Q(\eta_0,\overline{\eta_\infty})$} \label{S:eta0}

The nilpotency of the $N_j \in W_{-2}(\fg_\bC)$ implies that $\theta(t) \cdot e_0$ is polynomial in the $\ell(t_j)$.  From \eqref{E:g(H)} and \eqref{E:N} we see that this polynomial has degree at most $\sfm-\sfw$.  Write
\[
  \theta(t) \cdot e_0 \ = \ 
  \sum_{|a| \le\sfm-\sfw} 
  c_a\,\ell(t_1)^{a_1}\cdots\ell(t_k)^{a_k}\,
  N_1^{a_1}\cdots N_k^{a_k}(e_0) \,,
\]
with $a = (a_1,\ldots,a_k)$ is a $k$-tuple of non-negative integers, and $|a| = a_1+\cdots+a_k$.   By \eqref{E:eta-loc} and \eqref{E:eta-infty}, we have
\[
  Q \left( \eta_0 , \overline{\eta_\infty} \right)
  \ = \ \sum_{|a| \le\sfm-\sfw} 
  c_a\,\ell(t_1)^{a_1}\cdots\ell(t_k)^{a_k}\,
  Q \left( \hat\z(t) \cdot
  N_1^{a_1}\cdots N_k^{a_k}(e_0) \,,\,
  \overline{\hat\z(t) \cdot e_\infty}
  \right) \,.
\]

\subsubsection{Comment on computation of limits} \label{S:limcmt}

In order to prove Theorem \ref{T:extn} we need to show that 
\begin{equation}\label{E:limhI}
  \lim_{\substack{t_i\to0 \\ i\in I}} h \quad
  \hbox{\emph{exists and defines a smooth function on $Z^*_I \cap U$;}}
\end{equation}
and that the resulting function $h : U \to \bR$ is continuous.  We will prove \eqref{E:limhI} in the case that $I = J = \{1,\ldots,k\}$.  The general case $I \subset J$ is a straightforward generalization of the argument here, as is the exercise to verify that the resulting function is continuous on all of $U$; details are left to the reader.

Returning to \S\ref{S:eta0}, we are going to show that 
\begin{equation}\label{E:lim0}
  \lim_{t_1,\ldots,t_k\to 0} \ell(t_1)^{a_1}\cdots\ell(t_k)^{a_k}\,
  Q \left( \hat\z(t) \cdot
  N_1^{a_1}\cdots N_k^{a_k}(e_0) \,,\,
  \overline{\hat\z(t) \cdot e_\infty}
  \right) \ = \ 0 \,,
\end{equation}
whenever $|a|>0$.  Assume for the moment that \eqref{E:lim0} holds.  Then
\begin{equation}\label{E:limQ}
  \lim_{t_1,\ldots,t_k\to 0} 
  Q \left( \eta_0 , \overline{\eta_\infty} \right)
  \ = \ Q \left( \hat\z(t) \cdot e_0 \,,\,
  \overline{\hat\z(t) \cdot e_\infty}
  \right) \,.
\end{equation}
Since $\hat\z(t)$ is continuous on $U$, and holomorphic on strata $Z_I^* \cap U$ (\S\ref{S:a}), the desired \eqref{E:limhI} now follows (in the case $I =J$) by the definition \eqref{E:tildeh} of $\tilde h$: we have 
\begin{equation}\label{E:hZJ}
  h|_{Z_J^* \cap U} \ = \ \tRe\, Q \left( \hat\z(t) \cdot e_0 \,,\,
  \overline{\lambda \hat\z(t) \cdot e_\infty}
  \right)
\end{equation}  
This completes the proof of Theorem \ref{T:extn} (modulo the reader's exercise).

\subsubsection{Proof of \eqref{E:lim0}} \label{S:lim0}

The key lemma is the following.  Let $W^j = W(N_j)$ be the weight filtration along $Z_j^* \cap U$.

\begin{lemma} \label{L:Q0}
If $\hat\z(t)$ stabilizes \emph{any} one of the $W^j$ with $a_j > 0$, then 
\begin{equation}\label{E:Q0}
  Q \left( \hat\z(t) \cdot
  N_1^{a_1}\cdots N_k^{a_k}(e_0) \,,\,
  \overline{\hat\z(t) \cdot e_\infty}
  \right) \ = \ 0
\end{equation}
\end{lemma}

\begin{proof}
The lemma is a consequence of properties of weight filtrations.  By Lemma \ref{L:m}, there exist $\sfw \le \sfm_j \le \sfm$ so that
\begin{equation}\label{E:e0W}
  e_0 \ \in \ W^j_{\sfm_j}(H_\bC) \,, \quad
  e_0 \ \not\in \ W^j_{\sfm_j-1}(H_\bC) \,,
\end{equation}
and
\begin{equation}\label{E:einfty}
  e_\infty \ \in \ W^j_{2\sfw-\sfm_j}(H_\bC) \,,\quad
  e_\infty \ \not\in \ W^j_{2\sfw-\sfm_j-1}(H_\bC) \,.
\end{equation}
We have
\begin{equation}\label{E:Ne0}
  N_1^{a_1}\cdots N_k^{a_k}(e_0) \ \in \ 
  W^1_{\sfm_1-2a_1}(H_\bC) \,\cap\cdots\cap\, 
  W^k_{\sfm_k-2a_k}(H_\bC)\,.
\end{equation}
The essential property of the weight filtrations that we will utilize is that they are $Q$-isotropic
\begin{equation}\label{E:Wj}
  Q(W^j_\ell \,,\, W^j_m) \ = \ 0
  \qquad \forall \quad \ell+m< 2\,\sfw \,.
\end{equation}
The lemma now follows from \eqref{E:einfty}, \eqref{E:Ne0} and \eqref{E:Wj}.
\end{proof}

\begin{corollary} \label{C:Q0}
If $\hat\z(t)$ centralizes \emph{any} one of the $N_j$ with $a_j > 0$, then \eqref{E:Q0} holds.
\end{corollary}

\begin{proof}
The centralizer of $N_j$ preserves the weight filtration $W^j$.
\end{proof}

\subsubsection{Completing the proof of Theorem \ref{T:extn}}

By definition $f_I$ centralizes every $N_j$ with $j \in I$.  And since the $N_j$ all commute, the $\hat\theta_I$ centralizes every $N_j$.  So $\tAd_{\hat\theta_I(t)}(f_I)$ also centralizes every $N_j$ with $j \in I$.  Now, $\hat\z(t)$ will centralize $N_j$ if $f_I(t)=0$ for every $I \not\ni j$.  So \eqref{E:Q0} will hold unless, for every $j$ such that $a_j > 0$, there exists $I \not\ni j$ with $f_I(t)\not=0$.  Regard 
\[
  \hat\z(t) \ = \ \sum A_b\,t_1^{b_1} \cdots t_k^{b_k}
\]
as polynomial in the $t_j$ with coefficients taking value in $\tEnd(H)$.  The coefficient $A_b$ will centralize $N_j$ if $b_j=0$.  Rewrite the left-hand side of \eqref{E:Q0} as 
\[
  \sum_{b,c}
  t_1^{b_1} \overline{t}{}^{c_1}_1
  \cdots t_k^{b_k} \overline{t}{}^{c_k}_k \,
  Q \left( A_b \cdot
  N_1^{a_1}\cdots N_k^{a_k}(e_0) \,,\,
  \overline{A_c \cdot e_\infty}
  \right) \,,
\]
and note that 
$Q \left( A_b \cdot N_1^{a_1}\cdots N_k^{a_k}(e_0) \,,\,
  \overline{A_c \cdot e_\infty} \right) = 0$
if we have $a_j>0$ and $b_j+c_j=0$ for some $j$.  Returning to \eqref{E:lim0} we see that if $a_j>0$, then 
\[
  Q \left( \hat\z(t) \cdot N_1^{a_1}\cdots N_k^{a_k}(e_0) \,,\,
  \overline{\hat\z(t) \cdot e_\infty} \right)
\]
is a multiple of either $t_j$ or $\overline{t}{}_j$.  This establishes the desired \eqref{E:lim0}, and completes the proof of Theorem \ref{T:extn}.  \hfill \qed

\subsection{Constancy on fibres} \label{S:cnst}

\begin{theorem} \label{T:cnst}
The function $h : X \to \bR$ is constant on $\hPhiS$--fibres.
\end{theorem}

\begin{proof}
Because $h$ is defined by the period matrix representation, it is immediate that $h$ is locally constant on the fibres of $\Phi|_{B\cap X} = \PhiS|_{B \cap X}$, and therefore constant on the fibres of $\hPhiS|_{B\cap X}$.   

Suppose that $Z_J^* \cap A$ is nonempty.  Since the restriction of $h$ to $Z_J^* \cap X$ is a positive multiple of $h_J$ (Theorem \ref{T:hI}), it follows that $h$ is constant on the fibres of $\hPhiS|_{Z_J^* \cap X}$.

For the general case, recall the discussion of \S\S\ref{S:loc1}--\ref{S:loc2}.  The coordinate chart is centered at a point $b \in Z_J^* \cap X$.  The nilpotent operator $N = N_1+\cdots+N_k$ determines a weight filtration $W^J = W(N)$.  The fact that the restriction of $\z$ to $Z_J^* \cap U$ centralizes the nilpotent operators $N_j$ (\S\ref{S:Aloc}) implies that the restrictions of both $\z$ and $\hat\z$ to $Z_J^* \cap U$ preserve the weight filtration $W^J$.  By Lemma \ref{L:m} there exists $\sfw \le \sfm_J \le \sfm$ so that $e_0 \in W^J_{\sfm_J}(H_\bC)$ and $e_\infty \in W^J_{2\sfw-\sfm_J}(H_\bC)$.  The fact that $W^J$ is $Q$--isotropic, $Q(W^J_\ell , W^J_m) = 0$ for all $\ell+m < 2\,\sfw$, implies that the map $W^J_0(\fg_\bC) \to \bC$ given by $w \mapsto Q(\exp(w) \cdot e_0 , \overline{\exp(w) \cdot e_\infty})$ descends to a well-defined $W^J_0(\fg_\bC)/W^J_{-1}(\fg_\bC) \to \bC$.  It now follows from \eqref{E:hZJ} and the definition of $\Phi_J$ in \cite{GGR-part1} that $h$ is constant on the fibres of $\hPhiS$.
\end{proof}


\begin{corollary} \label{C:cnst}
The function $h : X \to \bR$ descends to a continuous function $h : \sX \to \bR$ on $\sX = \hPhiS(X) \subset \hatP$.
\end{corollary}

\subsection{Relationship to Hodge norms} \label{S:hI}

\begin{theorem} \label{T:hI}
Assume $Z_J^* \cap A$ is nonempty.  The restriction of $h$ to $Z_J^* \cap X$ is a positive multiple of $h_J$, where $h_J^\sfa$ is the Hodge norm-squared of $\Lambda^{\ot \sfa} \big|_{Z_J^* \cap A}$.
\end{theorem}

\noindent Since $h_J$ is the pull-back of a metric with negative curvature, cf.~\cite[(4.66)]{MR0229641}, we immediately obtain

\begin{corollary} \label{C:hI}
Assume $Z_J^* \cap A$ is nonempty. The restriction of $-\log h$ to $Z_J^* \cap X$ is psh.  Given $v \in T(Z_J^* \cap X)$, we have $-\ddb\log h(v,\overline v) = 0$ if and only if $v$ is tangent to a $\PhiS$--fibre (equivalently, a $\Phi_J$--fibre).
\end{corollary}

\begin{proof}[Proof of Theorem \ref{T:hI}]
Recall the discussion of \S\S\ref{S:loc1}--\ref{S:loc2}.  The coordinate chart is centered at a point $b \in Z_J^* \cap X$.  Assume that $Z_J^* \cap A$ is nonempty.  Set $N = N_1+\cdots+N_k$.  Then $W = W(N)$.  We claim that the Hodge norm-squared of $\Lambda^{\ot \sfa}\big|_{Z_J^* \cap U}$ is given by $h_J^\sfa$, where
\begin{equation}\label{E:hJ}
  h_J \ = \ 
  \bi^{2\sfw-\sfm} Q \left( \exp(f_J) \cdot e_0 \,,\, 
  N^{\sfm-\sfw} \,\overline{\exp(f_J) \cdot e_0} \right) \,.
\end{equation}
To see this, first recall that $\fz_J \subset W_0(\fg_\bC)$.  By construction $e_0 \in W_{\sfm}(H_\bC)$ and $e_\infty \in W_{2\sfw-\sfm}(H_\bC)$.  Then the fact that $W$ is $Q$--isotropic
\begin{equation}\label{E:QW}
  Q \left( W_\ell(H) \,,\, W_m(H) \right) \ = \ 0
  \qquad \forall \quad \ell+m < 2\,\sfw
\end{equation}
implies that the map $\fz_J \to \bC$ given by 
\[
  w \ \mapsto \ 
  Q \left( \exp(w) \cdot e_0 \,,\, 
  N^{\sfm-\sfw}\,\overline{\exp(w) \cdot e_0} \right)
\] 
descends to a well-defined $\fz_J/W_{-1}(\fz_J) \to \bC$.  Now to establish \eqref{E:hJ} it suffices to point out that $\{ f_J \hbox{ mod } W_{-1}(\fz_J) \}$ is the local period matrix representation of the period map $\Phi_J : Z_J^* \to \Gamma_J \bs D_J$.

Suppose for the moment that $b \in A$.  Then $f_J \equiv 0$ modulo $W_{-1}(\fz_J)$ along $Z_J^* \cap A \cap U$, and \eqref{E:hJ} implies
\begin{equation}\label{E:pos1}
  \bi^{2\sfw-\sfm} Q \left( e_0 \,,\, 
  N^{\sfm-\sfw}\,\overline{e_0} \right) \ > \ 0 \,.
\end{equation}
Since $f_J$ takes value in the centralizer $\fz_J$ of the $\{N_j\}_{j=1}^k$, it follows from Lemma \ref{L:hI} that
\begin{equation}\label{E:Q=hJ}
  Q \left( \exp(f_J) \cdot e_0 \,,\, 
  \overline{\lambda\,\exp(f_J) \cdot e_\infty} \right)
  \quad\hbox{\emph{is a positive multiple of}}\quad 
  h_J
\end{equation}
By \eqref{E:tildeh}, \eqref{E:lima} and \eqref{E:limQ}
\[
  Q \left( \exp(f_J) \cdot e_0 \,,\, 
  \overline{\lambda\,\exp(f_J) \cdot e_\infty} \right)
  \ = \ h\big|_{Z_J^* \cap U} \,.
\]
This establishes the theorem.
\end{proof}

\begin{lemma} \label{L:hI}
Recall the scalar $\lambda$ defined by \eqref{E:lambda}.  Suppose that $N \in W_{-2}(\fg_\bR)$ polarizes some mixed Hodge structure $(W,F')$ arising along $A$.  Then $\overline{\lambda \, e_\infty}$ is a positive multiple of $\bi^{2\sfw-\sfm}\, N^{\sfm-\sfw}\, \overline{e_0}$.
\end{lemma}

\begin{proof}
By Remark \ref{R:anotherchoice}, 
\[
  N \ \in \ \bigoplus_{p,q\le-1} \fg^{p,q}_{W,F} \,.
\]
Recall that $e_0 \in H^{\sfw,\sfm-\sfw}_{W,F}$.  Then \eqref{E:g(H)}, and \eqref{E:h} imply $N^{\sfm-\sfw} e_0 \in H^{2\sfw-\sfm,0}_{W,F}$.  It follows from $H^{2\sfw-\sfm,0}_{W,F} = \tspan_\bC\{e_\infty\}$ that $N^{\sfm-\sfw} e_0$ is a multiple of $e_\infty$.  By \eqref{E:Q} and \eqref{E:lambda} we have
\begin{equation} \nonumber 
  Q( e_0 \,,\, \overline{\lambda\,e_\infty} ) \ = \ 1 \,.
\end{equation}
Taken with \eqref{E:pos1}, this implies the lemma.
\end{proof}

\section{The hermitian symmetric case} \label{S:herm}

\begin{theorem} \label{T:herm}
If $D$ is hermitian, then the function $h : X \to \bR$ is smooth, and $-\log h$ is plurisubharmonic.
\end{theorem}

The essential point in the proof of Theorem \ref{T:herm} is the following lemma. 

\begin{lemma} \label{L:herm}
If $D$ is hermitian, then the subspace $\fs_F^\perp$ of \eqref{E:sFperp} centralizes the nilpotent elements $N_1,\ldots,N_k$ of \S\ref{S:loc1}.
\end{lemma}

\begin{proof}[Proof of Theorem \ref{T:herm}: smoothness]
By Lemma \ref{L:herm}, the function $\log \z : U \to \fs_F^\perp$ of \S\ref{S:loc2} takes value in the centralizer of the $N_1,\ldots,N_k$.  This implies that $\hat\z(t) = \z(t)$ is smooth.  And since $\hat\z(t)$ centralizes the $N_1,\ldots,N_k$ it follows from Corollary \ref{C:Q0} and \S\ref{S:eta0} that 
\begin{equation}\label{E:Qherm}
  Q(\eta_0,\overline{\eta_\infty}) \ = \ 
  Q \left( \hat\z(t)\cdot e_0 \,,\,
  \overline{\hat\z(t)\cdot e_\infty} \right)
\end{equation}
is smooth.  It now follows from the definition \eqref{E:tildeh} that $h$ is smooth.
\end{proof}

\begin{proof}[Proof of Theorem \ref{T:herm}: plurisubharmonicity]
By Lemma \ref{L:hI} and \eqref{E:Qherm} we have 
\[
  Q(\eta_0 , \overline{\lambda \eta_\infty}) \ = \ 
  \bi^{2\sfw-\sfm}\,Q \left( \hat\z(t)\cdot e_0 \,,\,
  N^{\sfm-\sfw}\,\overline {\hat\z(t)\cdot e_0} \right) \,.
\]
By \eqref{E:tildeh} and Theorem \ref{T:descent}
\[
  h \ = \ \bi^{2\sfw-\sfm}\,Q \left( \hat\z(t)\cdot e_0 \,,\,
  N^{\sfm-\sfw}\,\overline {\hat\z(t)\cdot e_0} \right) \,,
\]
modulo rescaling by a positive constant.  Since $N$ polarizes the mixed Hodge structure $(W,F)$, and $\eta(t)$ centralizes $N$, it follows that $N$ also polarizes the mixed Hodge structure $(W,\hat\z(t)\cdot F)$, \cite{MR3474815}.  This yields a strengthening of \eqref{E:Q=hJ}: $h$ is a Hodge metric, for the polarized Hodge structures $F^p(\tGr^W_\ell)$, on all of $X$.  Since this metric has nonpositive curvature \cite{MR0229641, MR0259958}, it follows that $-\log h$ is plurisubharmonic on all of $X$.
\end{proof}

\begin{proof}[Proof of Lemma \ref{L:herm}]
Then the Deligne splitting $\fg_\bC = \op\,\fg^{p,q}_{W,F}$ of \eqref{E:gpq} has the property that 
\[
  \fg^{p,q}_{W,F} \ = \ 0 
  \quad\hbox{ if either $|p| > 1$ or $|q|>1$.}
\]
By \eqref{E:dsfilts} we have 
\[
  W_{-2}(\fg_\bC) \ = \ \fg^{-1,-1}_{W,F} \,,
\]
and by \eqref{E:sFperp} we have 
\[
  \fs_F^\perp \ = \ \fg^{-1,1}_{W,F} \,\op\, \fg^{-1,0}_{W,F}
  \,\op\, \fg^{-1,-1}_{W,F} \,.
\]
It follows from \eqref{E:adg} that 
\[
  [ \fs_F^\perp \,,\, W_{-2}(\fg_\bC) ] \ = \ 0 \,.
\]
By \eqref{E:N}, the nilpotent operators $N_j$ of \S\ref{S:loc1} lie in $W_{-2}(\fg_\bC)$.   
\end{proof}

\section{Criterion for smoothness}

We have seen that $h : X \to \bR$ is smooth when $D$ is hermitian (Theorem \ref{T:herm}).  Smoothness is actually a consequence of the weaker condition that $\fs_F^\perp \subset W_0(\fg_\bC)$.

\begin{theorem} \label{T:sm}
If $\fs_F^\perp \subset W_0(\fg_\bC)$, then $h : X \to \bR$ is smooth.
\end{theorem}

\begin{remark}
The hypothesis that $\fs_F^\perp \subset W_0(\fg_\bC)$ can not be dropped, see \cite{Robles-pseudocnvx-eg}.
\end{remark}

\begin{proof}
As in the proof of Theorem \ref{T:herm}, it follows from Lemma \ref{L:Q0} and \S\ref{S:eta0} that 
\begin{equation} \nonumber 
  Q(\eta_0,\overline{\eta_\infty}) \ = \ 
  Q \left( \hat\z(t)\cdot e_0 \,,\,
  \overline{\hat\z(t)\cdot e_\infty} \right) \,.
\end{equation}

In general, $\hat\z(t)$ is not smooth.  However, $\log\z \in W_0(\fg_\bC)$ and $N_j \in W_{-2}(\fg_\bC)$ implies that $\log\hat\z(t)$ is smooth modulo $W_{-2}(\fg_\bC)$.

As in the proof of Theorem \ref{T:hI}, \eqref{E:e0W}, \eqref{E:einfty} and the fact that $W$ is $Q$--isotropic \eqref{E:QW} imply that the map $W_0(\fg_\bC) \to \bC$ given by 
\[
  w \ \mapsto \ 
  Q \left( \exp(w)\cdot e_0 \,,\,
  \overline{\exp(w)\cdot e_\infty} \right)
\]
descends to a well-defined map $W_0(\fg_\bC)/W_{-1}(\fg_\bC) \to \bC$.  It follows that $Q(\eta_0,\overline{\eta_\infty})$ is smooth.  Smoothness of $h$ follows from the definition \eqref{E:tildeh}.
\end{proof}

\appendix

\section{Hodge diamonds} \label{S:hd}

Given a mixed Hodge structure $(W,F)$ on a vector space $V$ the \emph{Hodge diamond} $\Diamond_{W,F}(V)$ is a visual representation of the Deligne splitting $V_\bC = \op\,V^{p,q}_{W,F}$ (\S\ref{S:gpq}) that is given by a configuration of points in the $(p,q)$--plane that are labeled with $\tdim_\bC\,V^{p,q}_{W,F}$.  This device encodes much of the discrete data in $(W,F)$, and may illuminate some of the constructions here that utilize limiting mixed Hodge structures.  In this appendix we consider two period domains, and list all possible Hodge diamonds coming from limiting mixed Hodge structures on those period domains.  We include the diamonds $\Diamond(\fg)$ and $\Diamond(H)$ of the induced mixed Hodge structures.  In the first example the domain is hermitian (\S\ref{S:hd1}), in the second example the domain is non-hermitian (\S\ref{S:hd2}).

\subsection{Weight $\sfn=1$ and $g=3$} \label{S:hd1}
Suppose that $D$ is the hermitian symmetric period domain parameterizing pure, effective, weight $\sfn=1$ polarized Hodge structures on $V \simeq \bQ^6$.  There are four possible Hodge diamonds, indexed by nonnegative integers $0 \le a,b \in\bZ$ satisfying $a+b=3$.  The diamonds for $V$ and $\fg$ are given by 
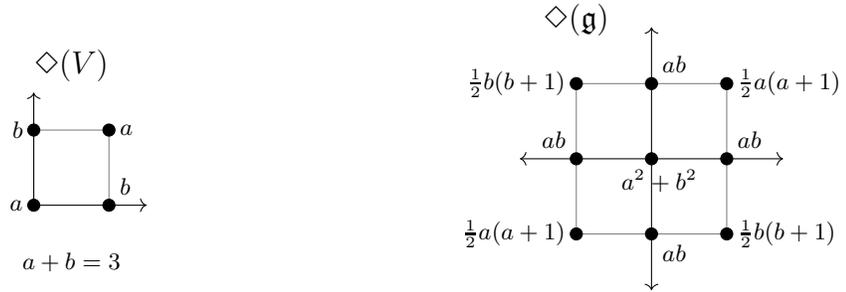
\begin{center}
\hsp{0pt} \hfill
\begin{tikzpicture}
  \node [above] at (0.5,1.5) {$\Diamond(V)$};
  \draw [<->] (0,1.5) -- (0,0) -- (1.5,0);
  \draw [gray] (1,0) -- (1,1);
  \draw [gray] (0,1) -- (1,1);
  \draw [fill] (0,0) circle [radius=0.08];
  \node [left] at (0,0) {\scriptsize{$a$}};
  \draw [fill] (0,1) circle [radius=0.08];
  \node [left] at (0,1) {\scriptsize{$b$}};
  \draw [fill] (1,1) circle [radius=0.08];
  \node [right] at (1,1) {\scriptsize{$a$}};
  \draw [fill] (1,0) circle [radius=0.08];
  \node [above right] at (1,0) {\scriptsize{$b$}};
  \node at (0.5,-0.75) {\scriptsize{$a+b=3$}};
  \node at (0,-1) {};
\end{tikzpicture}
\hfill %
\begin{tikzpicture}
  \node [above] at (-1,1.5) {$\Diamond(\fg)$};
  \draw [<->] (0,-1.75) -- (0,1.75);
  \draw [<->] (-1.75,0) -- (1.75,0);
  \draw [gray] (1,-1) -- (1,1);
  \draw [gray] (-1,-1) -- (-1,1);
  \draw [gray] (-1,1) -- (1,1);
  \draw [gray] (-1,-1) -- (1,-1);
  \draw [fill] (-1,1) circle [radius=0.08];
  \node [left] at (-1,1) {\scriptsize{$\half b(b+1)$}};
  \draw [fill] (-1,0) circle [radius=0.08];
  \node [above left] at (-1,0) {\scriptsize{$ab$}};
  \draw [fill] (-1,-1) circle [radius=0.08];
  \node [left] at (-1,-1) {\scriptsize{$\half a(a+1)$}};
  \draw [fill] (0,1) circle [radius=0.08];
  \node [above right] at (0,1) {\scriptsize{$ab$}};
  \draw [fill] (0,0) circle [radius=0.08];
  \node [below] at (0.1,0) {\scriptsize{$a^2+b^2$}};
  \draw [fill] (0,-1) circle [radius=0.08];
  \node [below right] at (0,-1) {\scriptsize{$ab$}};
  \draw [fill] (1,1) circle [radius=0.08];
  \node [right] at (1,1) {\scriptsize{$\half a(a+1)$}};
  \draw [fill] (1,0) circle [radius=0.08];
  \node [above right] at (1,0) {\scriptsize{$ab$}};
  \draw [fill] (1,-1) circle [radius=0.08];
  \node [right] at (1,-1) {\scriptsize{$\half b (b+1)$}};
\end{tikzpicture}
\end{center}
The underlying limiting mixed Hodge structure is pure ($\s=0$) if and and only if $a=0$.  The Hodge diamonds for $H = \tw^3 V$ (of weight $\sfw=3$) are 
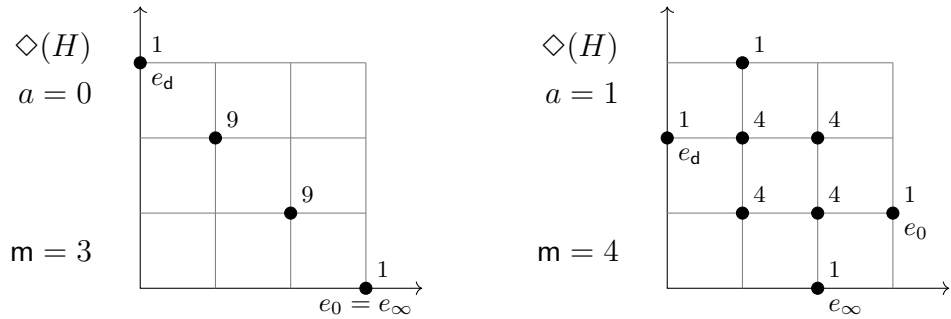
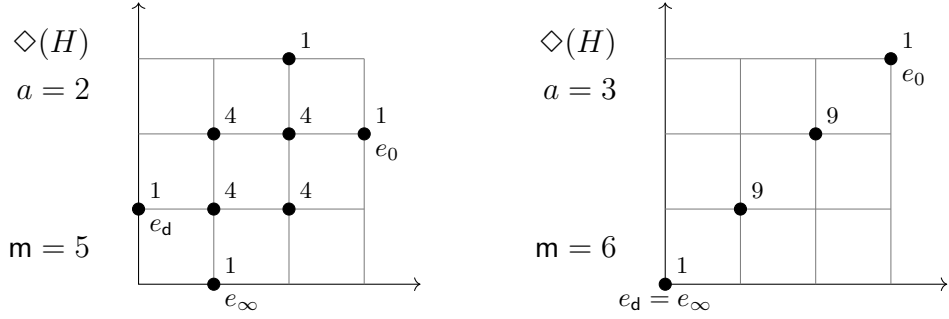
\begin{center}
\hsp{0pt} \hfill
\begin{tikzpicture}
  \node [left] at (-0.5,3.2) {$\Diamond(H)$};
  \node [left] at (-0.5,2.6) {$a=0$};
  \node [left] at (-0.5,0.5) {$\sfm=3$};
  \draw [<->] (0,3.75) -- (0,0) -- (3.75,0);
  \draw [gray] (1,0) -- (1,3);
  \draw [gray] (2,0) -- (2,3);
  \draw [gray] (3,0) -- (3,3);
  \draw [gray] (0,1) -- (3,1);
  \draw [gray] (0,2) -- (3,2);
  \draw [gray] (0,3) -- (3,3);
  \draw [fill] (0,3) circle [radius=0.08];
  \node [above right] at (0,3) {\scriptsize{$1$}};
  \node [below right] at (0,3) {\footnotesize{$e_\sfd$}};
  \draw [fill] (1,2) circle [radius=0.08];
  \node [above right] at (1,2) {\scriptsize{$9$}};
  \draw [fill] (2,1) circle [radius=0.08];
  \node [above right] at (2,1) {\scriptsize{$9$}};
  \draw [fill] (3,0) circle [radius=0.08];
  \node [above right] at (3,0) {\scriptsize{$1$}};
  \node [below] at (3,0) {\footnotesize{$e_0=e_\infty$}};
\end{tikzpicture}
\hfill %
\begin{tikzpicture}
  \node [left] at (-0.5,3.2) {$\Diamond(H)$};
  \node [left] at (-0.5,2.6) {$a=1$};
  \node [left] at (-0.5,0.5) {$\sfm=4$};
  \draw [<->] (0,3.75) -- (0,0) -- (3.75,0);
  \draw [gray] (1,0) -- (1,3);
  \draw [gray] (2,0) -- (2,3);
  \draw [gray] (3,0) -- (3,3);
  \draw [gray] (0,1) -- (3,1);
  \draw [gray] (0,2) -- (3,2);
  \draw [gray] (0,3) -- (3,3);
  \draw [fill] (0,2) circle [radius=0.08];
  \node [above right] at (0,2) {\scriptsize{$1$}};
  \node [below right] at (0,2) {\footnotesize{$e_\sfd$}};
  \draw [fill] (1,3) circle [radius=0.08];
  \node [above right] at (1,3) {\scriptsize{$1$}};
  \draw [fill] (1,2) circle [radius=0.08];
  \node [above right] at (1,2) {\scriptsize{$4$}};
  \draw [fill] (1,1) circle [radius=0.08];
  \node [above right] at (1,1) {\scriptsize{$4$}};
  \draw [fill] (2,2) circle [radius=0.08];
  \node [above right] at (2,2) {\scriptsize{$4$}};
  \draw [fill] (2,1) circle [radius=0.08];
  \node [above right] at (2,1) {\scriptsize{$4$}};
  \draw [fill] (2,0) circle [radius=0.08];
  \node [above right] at (2,0) {\scriptsize{$1$}};
  \node [below right] at (2,0) {\footnotesize{$e_\infty$}};
  \draw [fill] (3,1) circle [radius=0.08];
  \node [above right] at (3,1) {\scriptsize{$1$}};
  \node [below right] at (3,1) {\footnotesize{$e_0$}};
\end{tikzpicture}
\hfill \hsp{0pt} \\
\hsp{0pt} \hfill
\begin{tikzpicture}
  \node [left] at (-0.5,3.2) {$\Diamond(H)$};
  \node [left] at (-0.5,2.6) {$a=2$};
  \node [left] at (-0.5,0.5) {$\sfm=5$};
  \draw [<->] (0,3.75) -- (0,0) -- (3.75,0);
  \draw [gray] (1,0) -- (1,3);
  \draw [gray] (2,0) -- (2,3);
  \draw [gray] (3,0) -- (3,3);
  \draw [gray] (0,1) -- (3,1);
  \draw [gray] (0,2) -- (3,2);
  \draw [gray] (0,3) -- (3,3);
  \draw [fill] (0,1) circle [radius=0.08];
  \node [above right] at (0,1) {\scriptsize{$1$}};
  \node [below right] at (0,1) {\footnotesize{$e_\sfd$}};
  \draw [fill] (1,2) circle [radius=0.08];
  \node [above right] at (1,2) {\scriptsize{$4$}};
  \draw [fill] (1,1) circle [radius=0.08];
  \node [above right] at (1,1) {\scriptsize{$4$}};
  \draw [fill] (1,0) circle [radius=0.08];
  \node [above right] at (1,0) {\scriptsize{$1$}};
  \node [below right] at (1,0) {\footnotesize{$e_\infty$}};
  \draw [fill] (2,3) circle [radius=0.08];
  \node [above right] at (2,3) {\scriptsize{$1$}};
  \draw [fill] (2,2) circle [radius=0.08];
  \node [above right] at (2,2) {\scriptsize{$4$}};
  \draw [fill] (2,1) circle [radius=0.08];
  \node [above right] at (2,1) {\scriptsize{$4$}};
  \draw [fill] (3,2) circle [radius=0.08];
  \node [above right] at (3,2) {\scriptsize{$1$}};
  \node [below right] at (3,2) {\footnotesize{$e_0$}};
\end{tikzpicture}
\hfill %
\begin{tikzpicture}
  \node [left] at (-0.5,3.2) {$\Diamond(H)$};
  \node [left] at (-0.5,2.6) {$a=3$};
  \node [left] at (-0.5,0.5) {$\sfm=6$};
  \draw [<->] (0,3.75) -- (0,0) -- (3.75,0);
  \draw [gray] (1,0) -- (1,3);
  \draw [gray] (2,0) -- (2,3);
  \draw [gray] (3,0) -- (3,3);
  \draw [gray] (0,1) -- (3,1);
  \draw [gray] (0,2) -- (3,2);
  \draw [gray] (0,3) -- (3,3);
  \draw [fill] (3,3) circle [radius=0.08];
  \node [above right] at (3,3) {\scriptsize{$1$}};
  \node [below right] at (3,3) {\footnotesize{$e_0$}};
  \draw [fill] (2,2) circle [radius=0.08];
  \node [above right] at (2,2) {\scriptsize{$9$}};
  \draw [fill] (1,1) circle [radius=0.08];
  \node [above right] at (1,1) {\scriptsize{$9$}};
  \draw [fill] (0,0) circle [radius=0.08];
  \node [above right] at (0,0) {\scriptsize{$1$}};
  \node [below] at (0,0) {\footnotesize{$e_\sfd=e_\infty$}};
\end{tikzpicture}
\hfill \hsp{0pt} %
\end{center}

\subsection{Weight $\sfn=2$ and $\bfh(2,\sfh,2)$} \label{S:hd2}

Suppose that $D$ is the (non-hermitian) period domain parameterizing pure, effective, weight $\sfn=2$ polarized Hodge structures on $V$ with Hodge numbers $\bfh = (2,\sfh,2)$.  There are six possible Hodge diamonds.  We have $H = \tw^2V = \fg \ot \bQ(-2)$ and $\sfw=4$.  In the diamonds below, some of the nodes are left unmarked; those missing dimensions may be determined by \eqref{E:dsfilts} and \eqref{E:symmetries}.
\begin{equation} \label{E:hd2-0}
\begin{tikzpicture}[baseline={([yshift=-.5ex]current bounding box.center)}]
  \node [above] at (1.5,2.4) {$\Diamond(V)$}; 		
  \draw [<->] (0,2.75) -- (0,0) -- (2.75,0);		
  \node [left] at (2,-1) {$\sfm=4$};				
  \draw [gray] (1,0) -- (1,2);
  \draw [gray] (2,0) -- (2,2);
  \draw [gray] (0,1) -- (2,1);
  \draw [gray] (0,2) -- (2,2);
  \draw [fill] (0,2) circle [radius=0.08];
  \node [above right] at (0,2) {\scriptsize{$2$}};
  \draw [fill] (1,1) circle [radius=0.08];
  \node [above right] at (1,1) {\scriptsize{$\sfh$}};
  \draw [fill] (2,0) circle [radius=0.08];
  \node [above right] at (2,0) {\scriptsize{$2$}};
\end{tikzpicture}
\hsp{70pt}
\begin{tikzpicture}[baseline={([yshift=-.5ex]current bounding box.center)}]
  \node [above] at (1,2.5) {$\Diamond(H)$};			
  \draw [<->] (-2,2.75) -- (-2,-2) -- (2.75,-2);	
  \draw [gray] (0,-2) -- (0,2);						
  \draw [gray] (-2,0) -- (2,0);						
  \draw [gray] (2,-2) -- (2,2);
  \draw [gray] (1,-2) -- (1,2);
  \draw [gray] (-2,-2) -- (-2,2);
  \draw [gray] (-1,-2) -- (-1,2);
  \draw [gray] (-2,2) -- (2,2);
  \draw [gray] (-2,1) -- (2,1);
  \draw [gray] (-2,-1) -- (2,-1);
  \draw [gray] (-2,-2) -- (2,-2);
  \draw [fill] (-2,2) circle [radius=0.08];
  \node [above right] at (-2,2) {\scriptsize{$1$}};
  \node [below right] at (-2,2) {\footnotesize{$e_\sfd$}};
  \draw [fill] (-1,1) circle [radius=0.08];
  \node [above right] at (-1,1) {\scriptsize{$2\sfh$}};
  \draw [fill] (0,0) circle [radius=0.08];
  \node [above right] at (0,0) {};
  \draw [fill] (1,-1) circle [radius=0.08];
  \node [above right] at (1,-1) {\scriptsize{$2\sfh$}};
  \draw [fill] (2,-2) circle [radius=0.08];
  \node [above right] at (2,-2) {\scriptsize{$1$}};
  \node [below] at (2,-2) {\footnotesize{$e_0=e_\infty$}};
\end{tikzpicture}
\end{equation}
\begin{equation} \label{E:hd2-1}
\begin{tikzpicture}[baseline={([yshift=-.5ex]current bounding box.center)}]
  \node [above] at (1.5,2.4) {$\Diamond(V)$}; 		
  \draw [<->] (0,2.75) -- (0,0) -- (2.75,0);		
  \node [left] at (2,-1) {$\sfm=5$};				
  \draw [gray] (1,0) -- (1,2);
  \draw [gray] (2,0) -- (2,2);
  \draw [gray] (0,1) -- (2,1);
  \draw [gray] (0,2) -- (2,2);
  \draw [fill] (0,2) circle [radius=0.08];
  \node [above right] at (0,2) {\scriptsize{$1$}};
  \draw [fill] (0,1) circle [radius=0.08];
  \node [above right] at (0,1) {\scriptsize{$1$}};
  \draw [fill] (1,2) circle [radius=0.08];
  \node [above right] at (1,2) {\scriptsize{$1$}};
  \draw [fill] (1,1) circle [radius=0.08];
  \node [above] at (1,1) {};
  \draw [fill] (1,0) circle [radius=0.08];
  \node [above right] at (1,0) {\scriptsize{$1$}};
  \draw [fill] (2,1) circle [radius=0.08];
  \node [above right] at (2,1) {\scriptsize{$1$}};
  \draw [fill] (2,0) circle [radius=0.08];
  \node [above right] at (2,0) {\scriptsize{$1$}};
\end{tikzpicture}
\hsp{70pt}
\begin{tikzpicture}[baseline={([yshift=-.5ex]current bounding box.center)}]
  \node [above] at (1,2.5) {$\Diamond(H)$};			
  \draw [<->] (-2,2.75) -- (-2,-2) -- (2.75,-2);	
  \draw [gray] (0,-2) -- (0,2);						
  \draw [gray] (-2,0) -- (2,0);						
  \draw [gray] (2,-2) -- (2,2);
  \draw [gray] (1,-2) -- (1,2);
  \draw [gray] (-2,-2) -- (-2,2);
  \draw [gray] (-1,-2) -- (-1,2);
  \draw [gray] (-2,2) -- (2,2);
  \draw [gray] (-2,1) -- (2,1);
  \draw [gray] (-2,-1) -- (2,-1);
  \draw [gray] (-2,-2) -- (2,-2);
  \draw [fill] (-2,1) circle [radius=0.08];
  \node [above right] at (-2,1) {\scriptsize{$1$}};
  \node [below right] at (-2,1) {\footnotesize{$e_\sfd$}};
  \draw [fill] (-1,2) circle [radius=0.08];
  \node [above right] at (-1,2) {\scriptsize{$1$}};
  \draw [fill] (-1,1) circle [radius=0.08];
  \draw [fill] (-1,0) circle [radius=0.08];
  \node [above] at (-1,0) {\scriptsize{$\sfh-1$}};
  \draw [fill] (-1,-1) circle [radius=0.08];
  \node [above right] at (-1,-1) {\scriptsize{$1$}};
  \draw [fill] (0,1) circle [radius=0.08];
  \node [above right] at (0,1) {};
  \draw [fill] (0,0) circle [radius=0.08];
  \node [above right] at (0,0) {};
  \draw [fill] (0,-1) circle [radius=0.08];
  \node [above right] at (0,-1) {};
  \draw [fill] (1,-1) circle [radius=0.08];
  \node [above] at (1,-1) {};
  \draw [fill] (1,0) circle [radius=0.08];
  \node [above] at (1,0) {};
  \draw [fill] (1,1) circle [radius=0.08];
  \node [above right] at (1,1) {\scriptsize{$1$}};
  \draw [fill] (1,-2) circle [radius=0.08];
  \node [above right] at (1,-2) {\scriptsize{$1$}};
  \node [below right] at (1,-2) {\footnotesize{$e_\infty$}};
  \draw [fill] (2,-1) circle [radius=0.08];
  \node [above right] at (2,-1) {\scriptsize{$1$}};
  \node [below right] at (2,-1) {\footnotesize{$e_0$}};
\end{tikzpicture}
\end{equation}
\begin{equation} \label{E:hd2-2}
\begin{tikzpicture}[baseline={([yshift=-.5ex]current bounding box.center)}]
  \node [above] at (1.5,2.4) {$\Diamond(V)$}; 		
  \draw [<->] (0,2.75) -- (0,0) -- (2.75,0);		
  \node [left] at (2,-1) {$\sfm=6$};				
  \draw [gray] (1,0) -- (1,2);
  \draw [gray] (2,0) -- (2,2);
  \draw [gray] (0,1) -- (2,1);
  \draw [gray] (0,2) -- (2,2);
  \draw [fill] (0,2) circle [radius=0.08];
  \node [above right] at (0,2) {\scriptsize{$1$}};
  \draw [fill] (0,0) circle [radius=0.08];
  \node [above right] at (0,0) {\scriptsize{$1$}};
  \draw [fill] (1,1) circle [radius=0.08];
  \node [above right] at (1,1) {\scriptsize{$\sfh$}};
  \draw [fill] (2,2) circle [radius=0.08];
  \node [above right] at (2,2) {\scriptsize{$1$}};
  \draw [fill] (2,0) circle [radius=0.08];
  \node [above right] at (2,0) {\scriptsize{$1$}};
\end{tikzpicture}
\hsp{70pt}
\begin{tikzpicture}[baseline={([yshift=-.5ex]current bounding box.center)}]
  \node [above] at (1,2.5) {$\Diamond(H)$};			
  \draw [<->] (-2,2.75) -- (-2,-2) -- (2.75,-2);	
  \draw [gray] (0,-2) -- (0,2);						
  \draw [gray] (-2,0) -- (2,0);						
  \draw [gray] (2,-2) -- (2,2);
  \draw [gray] (1,-2) -- (1,2);
  \draw [gray] (-2,-2) -- (-2,2);
  \draw [gray] (-1,-2) -- (-1,2);
  \draw [gray] (-2,2) -- (2,2);
  \draw [gray] (-2,1) -- (2,1);
  \draw [gray] (-2,-1) -- (2,-1);
  \draw [gray] (-2,-2) -- (2,-2);
  \draw [fill] (2,0) circle [radius=0.08];
  \node [above right] at (2,0) {\scriptsize{$1$}};
  \node [below right] at (2,0) {\footnotesize{$e_0$}};
  \draw [fill] (0,2) circle [radius=0.08];
  \node [above right] at (0,2) {\scriptsize{$1$}};
  \draw [fill] (-1,1) circle [radius=0.08];
  \node [above right] at (-1,1) {\scriptsize{$\sfh$}};
  \draw [fill] (-1,-1) circle [radius=0.08];
  \node [above right] at (-1,-1) {\scriptsize{$\sfh$}};
  \draw [fill] (0,0) circle [radius=0.08];
  \node [above right] at (0,0) {};
  \draw [fill] (1,-1) circle [radius=0.08];
  \node [above right] at (1,-1) {\scriptsize{$\sfh$}};
  \draw [fill] (1,1) circle [radius=0.08];
  \node [above right] at (1,1) {\scriptsize{$\sfh$}};
  \draw [fill] (0,-2) circle [radius=0.08];
  \node [above right] at (0,-2) {\scriptsize{$1$}};
  \node [below right] at (0,-2) {\footnotesize{$e_\infty$}};
  \draw [fill] (-2,0) circle [radius=0.08];
  \node [above right] at (-2,0) {\scriptsize{$1$}};
  \node [below right] at (-2,0) {\footnotesize{$e_\sfd$}};
\end{tikzpicture}
\end{equation}
\begin{equation} \label{E:hd2-3}
\begin{tikzpicture}[baseline={([yshift=-.5ex]current bounding box.center)}]
  \node [above] at (1.5,2.4) {$\Diamond(V)$}; 		
  \draw [<->] (0,2.75) -- (0,0) -- (2.75,0);		
  \node [left] at (2,-1) {$\sfm=6$};				
  \draw [gray] (1,0) -- (1,2);
  \draw [gray] (2,0) -- (2,2);
  \draw [gray] (0,1) -- (2,1);
  \draw [gray] (0,2) -- (2,2);
  \draw [fill] (0,1) circle [radius=0.08];
  \node [above right] at (0,1) {\scriptsize{$2$}};
  \draw [fill] (1,2) circle [radius=0.08];
  \node [above right] at (1,2) {\scriptsize{$2$}};
  \draw [fill] (1,1) circle [radius=0.08];
  \node [above] at (1,1) {};
  \draw [fill] (1,0) circle [radius=0.08];
  \node [above right] at (1,0) {\scriptsize{$2$}};
  \draw [fill] (2,1) circle [radius=0.08];
  \node [above right] at (2,1) {\scriptsize{$2$}};
\end{tikzpicture}
\hsp{70pt}
\begin{tikzpicture}[baseline={([yshift=-.5ex]current bounding box.center)}]
  \node [above] at (1,2.5) {$\Diamond(H)$};			
  \draw [<->] (-2,2.75) -- (-2,-2) -- (2.75,-2);	
  \draw [gray] (0,-2) -- (0,2);						
  \draw [gray] (-2,0) -- (2,0);						
  \draw [gray] (2,-2) -- (2,2);
  \draw [gray] (1,-2) -- (1,2);
  \draw [gray] (-2,-2) -- (-2,2);
  \draw [gray] (-1,-2) -- (-1,2);
  \draw [gray] (-2,2) -- (2,2);
  \draw [gray] (-2,1) -- (2,1);
  \draw [gray] (-2,-1) -- (2,-1);
  \draw [gray] (-2,-2) -- (2,-2);
  \draw [fill] (2,0) circle [radius=0.08];
  \node [above right] at (2,0) {\scriptsize{$1$}};
  \node [below right] at (2,0) {\footnotesize{$e_0$}};
  \draw [fill] (0,2) circle [radius=0.08];
  \node [above right] at (0,2) {\scriptsize{$1$}};
  \draw [fill] (-1,1) circle [radius=0.08];
  \node [above right] at (-1,1) {\scriptsize{$4$}};
  \draw [fill] (-1,0) circle [radius=0.08];
  \node [above left] at (-1,0) {};
  \draw [fill] (-1,-1) circle [radius=0.08];
  \node [above right] at (-1,-1) {\scriptsize{$4$}};
  \draw [fill] (0,1) circle [radius=0.08];
  \node [above right] at (0,1) {};
  \draw [fill] (0,0) circle [radius=0.08];
  \node [above right] at (0,0) {};
  \draw [fill] (0,-1) circle [radius=0.08];
  \node [above right] at (0,-1) {};
  \draw [fill] (1,-1) circle [radius=0.08];
  \node [above right] at (1,-1) {\scriptsize{$4$}};
  \draw [fill] (1,0) circle [radius=0.08];
  \node [above] at (1,0) {};
  \draw [fill] (1,1) circle [radius=0.08];
  \node [above right] at (1,1) {\scriptsize{$4$}};
  \draw [fill] (0,-2) circle [radius=0.08];
  \node [above right] at (0,-2) {\scriptsize{$1$}};
  \node [below right] at (0,-2) {\footnotesize{$e_\infty$}};
  \draw [fill] (-2,0) circle [radius=0.08];
  \node [above right] at (-2,0) {\scriptsize{$1$}};
  \node [below right] at (-2,0) {\footnotesize{$e_\sfd$}};
\end{tikzpicture}
\end{equation}
\begin{equation} \label{E:hd2-4}
\begin{tikzpicture}[baseline={([yshift=-.5ex]current bounding box.center)}]
  \node [above] at (1.5,2.4) {$\Diamond(V)$}; 		
  \draw [<->] (0,2.75) -- (0,0) -- (2.75,0);		
  \node [left] at (2,-1) {$\sfm=7$};				
  \draw [gray] (1,0) -- (1,2);
  \draw [gray] (2,0) -- (2,2);
  \draw [gray] (0,1) -- (2,1);
  \draw [gray] (0,2) -- (2,2);
  \draw [fill] (0,1) circle [radius=0.08];
  \node [above right] at (0,1) {\scriptsize{$1$}};
  \draw [fill] (0,0) circle [radius=0.08];
  \node [above right] at (0,0) {\scriptsize{$1$}};
  \draw [fill] (1,2) circle [radius=0.08];
  \node [above right] at (1,2) {\scriptsize{$1$}};
  \draw [fill] (1,1) circle [radius=0.08];
  \node [above] at (1,1) {};
  \draw [fill] (1,0) circle [radius=0.08];
  \node [above right] at (1,0) {\scriptsize{$1$}};
  \draw [fill] (2,2) circle [radius=0.08];
  \node [above right] at (2,2) {\scriptsize{$1$}};
  \draw [fill] (2,1) circle [radius=0.08];
  \node [above right] at (2,1) {\scriptsize{$1$}};
\end{tikzpicture}
\hsp{70pt}
\begin{tikzpicture}[baseline={([yshift=-.5ex]current bounding box.center)}]
  \node [above] at (1,2.5) {$\Diamond(H)$};			
  \draw [<->] (-2,2.75) -- (-2,-2) -- (2.75,-2);	
  \draw [gray] (0,-2) -- (0,2);						
  \draw [gray] (-2,0) -- (2,0);						
  \draw [gray] (2,-2) -- (2,2);
  \draw [gray] (1,-2) -- (1,2);
  \draw [gray] (-2,-2) -- (-2,2);
  \draw [gray] (-1,-2) -- (-1,2);
  \draw [gray] (-2,2) -- (2,2);
  \draw [gray] (-2,1) -- (2,1);
  \draw [gray] (-2,-1) -- (2,-1);
  \draw [gray] (-2,-2) -- (2,-2);
  \draw [fill] (2,1) circle [radius=0.08];
  \node [above right] at (2,1) {\scriptsize{$1$}};
  \node [below right] at (2,1) {\footnotesize{$e_0$}};
  \draw [fill] (1,2) circle [radius=0.08];
  \node [above right] at (1,2) {\scriptsize{$1$}};
  \draw [fill] (-1,1) circle [radius=0.08];
  \node [above right] at (-1,1) {\scriptsize{$1$}};
  \draw [fill] (-1,0) circle [radius=0.08];
  \node [above] at (-1,0) {\scriptsize{$\sfh-1$}};
  \draw [fill] (-1,-1) circle [radius=0.08];
  \draw [fill] (0,1) circle [radius=0.08];
  \node [above right] at (0,1) {};
  \draw [fill] (0,0) circle [radius=0.08];
  \node [above right] at (0,0) {};
  \draw [fill] (0,-1) circle [radius=0.08];
  \node [above right] at (0,-1) {};
  \draw [fill] (1,-1) circle [radius=0.08];
  \node [above] at (1,-1) {};
  \draw [fill] (1,0) circle [radius=0.08];
  \node [above] at (1,0) {};
  \draw [fill] (1,1) circle [radius=0.08];
  \node [above] at (1,1) {};
  \draw [fill] (-1,-2) circle [radius=0.08];
  \node [above right] at (-1,-2) {\scriptsize{$1$}};
  \node [below right] at (-1,-2) {\footnotesize{$e_\infty$}};
  \draw [fill] (-2,-1) circle [radius=0.08];
  \node [above right] at (-2,-1) {\scriptsize{$1$}};
  \node [below right] at (-2,-1) {\footnotesize{$e_\sfd$}};
\end{tikzpicture}
\end{equation}
\begin{equation}\label{E:hd2-5}
\begin{tikzpicture}[baseline={([yshift=-.5ex]current bounding box.center)}]
  \node [above] at (1.5,2.4) {$\Diamond(V)$}; 		
  \draw [<->] (0,2.75) -- (0,0) -- (2.75,0);		
  \node [left] at (2,-1) {$\sfm=8$};				
  \draw [gray] (1,0) -- (1,2);
  \draw [gray] (2,0) -- (2,2);
  \draw [gray] (0,1) -- (2,1);
  \draw [gray] (0,2) -- (2,2);
  \draw [fill] (0,0) circle [radius=0.08];
  \node [above right] at (0,0) {\scriptsize{$2$}};
  \draw [fill] (1,1) circle [radius=0.08];
  \node [above right] at (1,1) {\scriptsize{$\sfh$}};
  \draw [fill] (2,2) circle [radius=0.08];
  \node [above right] at (2,2) {\scriptsize{$2$}};
\end{tikzpicture}
\hsp{70pt}
\begin{tikzpicture}[baseline={([yshift=-.5ex]current bounding box.center)}]
  \node [above] at (1,2.5) {$\Diamond(H)$};			
  \draw [<->] (-2,2.75) -- (-2,-2) -- (2.75,-2);	
  \draw [gray] (0,-2) -- (0,2);						
  \draw [gray] (-2,0) -- (2,0);						
  \draw [gray] (2,-2) -- (2,2);
  \draw [gray] (1,-2) -- (1,2);
  \draw [gray] (-2,-2) -- (-2,2);
  \draw [gray] (-1,-2) -- (-1,2);
  \draw [gray] (-2,2) -- (2,2);
  \draw [gray] (-2,1) -- (2,1);
  \draw [gray] (-2,-1) -- (2,-1);
  \draw [gray] (-2,-2) -- (2,-2);
  \draw [fill] (-2,-2) circle [radius=0.08];
  \node [above right] at (-2,-2) {\scriptsize{$1$}};
  \node [below] at (-2,-2) {\footnotesize{$e_\sfd=e_\infty$}};
  \draw [fill] (-1,-1) circle [radius=0.08];
  \node [above right] at (-1,-1) {\scriptsize{$2\sfh$}};
  \draw [fill] (0,0) circle [radius=0.08];
  \node [above right] at (0,0) {};
  \draw [fill] (1,1) circle [radius=0.08];
  \node [above right] at (1,1) {\scriptsize{$2\sfh$}};
  \draw [fill] (2,2) circle [radius=0.08];
  \node [above right] at (2,2) {\scriptsize{$1$}};
  \node [below right] at (2,2) {\footnotesize{$e_0$}};
\end{tikzpicture}
\end{equation}

\section{Weight filtration lemma}

We will need, at two points in the paper, the following technical lemma on weight filtrations.

Let $D$ be a period domain (or Mumford--Tate domain) parameterizing pure, $Q$--polarized Hodge structures on a rational vector space $H$.  Suppose that the Hodge structures are of weight $\sfw$ and have first Hodge number $h^{\sfw,0} = 1$.  Let $(W,F)$ be a limiting Hodge structure polarized by a cone $\s = \tspan_{\bR_{>0}}\{N_1,N_2\}$ of commuting nilpotent operators.  Fix a basis $\{e_0,e_1,\ldots,e_\sfd\}$ of $H_\bC$ as in \S\ref{S:framing}, and define $e_\infty$ and $\sfm$ as in \S\ref{S:einfty}.

We have $W = W(N)[-\sfw]$, cf.~\cite[(2.26)]{MR840721}.  Let $W^1 = W(N_1)[-\sfw]$.  

\begin{lemma} \label{L:m}
There exists $\sfw \le \sfm_1 \le \sfm$ so that $e_0 \in W^1_{\sfm_1}(H_\bC)$ but $e_0 \not\in W^1_{\sfm_1-1}(H_\bC)$; and $e_\infty \in W^1_{2\sfw-\sfm_1}(H_\bC)$ but $e_\infty \not\in W^1_{2\sfw-\sfm_1-1}(H_\bC)$.
\end{lemma}

\begin{proof}
The several-variable $\tSL_2$--orbit theorem \cite{MR840721} associates to $(W,F)$ and the ordered pair $\{N_1,N_2\}$ a mixed Hodge structure $(W,\tilde F)$ and a pair of commuting $\fsl_2$'s with nilpotent operators $\{N_1 , \hat N_2\}$ so that: 
\begin{i_list}
\item
the cone $\s$ polarizes $(W,\tilde F)$; 
\item \label{i:hat}
the cone $\hat\s = \tspan_{\bR_{>0}}\{N_1,\hat N_2\}$ polarizes $(W,\tilde F)$; and 
\item
$\tilde F = g \cdot F$ for some $g \in G_\bC$ which commutes with both $N$ and $N_1$.  
\end{i_list}
The second item implies $W = W(N_1+\hat N_2)[-\sfw]$.  The third item implies $g$ preserves both weight filtrations $W$ and $W^1$.  So without loss of generality, we may assume that $F = \tilde F$.    Then \ref{i:hat} implies $e_\infty$ is a nonzero multiple of $\hat N{}^{\sfm-\sfw} e_0$.  

Since $N_1,\hat N_2 \in \fg^{-1,-1}_{W,F}$, $e_0 \in F^\sfw(H_\bC)$ is necessarily a highest weight vector of both $\fsl_2$'s.  Now it follows from standard $\fsl_2$--representation theory that there exists $\sfw \le \sfm_1 \le \sfm$ such that $e_0 \in W^1_{\sfm_1}(H_\bC)$ and $e_0 \not\in W^1_{\sfm_1-1}(H_\bC)$; and $\hat N{}^{\sfm-\sfw} e_0 = \hat N{}_2^{\sfm-\sfm_1}\,N_1^{\sfm_1-\sfw} e_0$.  The latter is an element of $W^1_{2\sfw-\sfm_1}(H_\bC)$, and so implies $e_\infty \in W^1_{2\sfw-\sfm_1}(H_\bC)$.

Finally $Q(e_0,\overline{e_\infty}) \not=0$ and the fact the the weight polarization $W^1$ is $Q$--isotropic
\[
  Q(W^1_k(H) , W^1_\ell(H)) \ = \ 0
  \quad\hbox{ for all } \quad k+\ell < \sfw\,,
\]
implies $e_\infty \not \in W^1_{2\sfw-\sfm_1-1}(H_\bC)$.
\end{proof}


\begin{thebibliography}{GGLR23}

\bibitem[BB66]{MR0216035}
W.~L. Baily, Jr. and A.~Borel.
\newblock Compactification of arithmetic quotients of bounded symmetric
  domains.
\newblock {\em Ann. of Math. (2)}, 84:442--528, 1966.

\bibitem[BBT23]{MR4557401}
Benjamin Bakker, Yohan Brunebarbe, and Jacob Tsimerman.
\newblock o-minimal {GAGA} and a conjecture of {G}riffiths.
\newblock {\em Invent. Math.}, 232(1):163--228, 2023.

\bibitem[BFMT25]{BFMT}
Benjamin Bakker, Stefano Filipazzi, Mirko Mauri, Jacob Tsimerman.
\newblock Baily--Borel compactifications of period images and the b-semiampless conjecture.
\newblock arXiv:2508.19215, 2025.

\bibitem[CKS86]{MR840721}
Eduardo Cattani, Aroldo Kaplan, and Wilfried Schmid.
\newblock Degeneration of {H}odge structures.
\newblock {\em Ann. of Math. (2)}, 123(3):457--535, 1986.

\bibitem[Del97]{MR1416353}
P.~Deligne.
\newblock Local behavior of {H}odge structures at infinity.
\newblock In {\em Mirror symmetry, {II}}, volume~1, pages 683--699. Amer. Math.
  Soc., Providence, RI, 1997.

\bibitem[Dem00]{MR1782659}
Jean-Pierre Demailly.
\newblock On the {O}hsawa-{T}akegoshi-{M}anivel {$L^2$} extension theorem.
\newblock In {\em Complex analysis and geometry ({P}aris, 1997)}, volume 188 of
  {\em Progr. Math.}, pages 47--82. Birkh\"{a}user, Basel, 2000.

\bibitem[Dem16]{MR3525916}
Jean-Pierre Demailly.
\newblock Extension of holomorphic functions defined on non reduced analytic
  subvarieties.
\newblock In {\em The legacy of {B}ernhard {R}iemann after one hundred and
  fifty years. {V}ol. {I}}, volume~35 of {\em Adv. Lect. Math. (ALM)}, pages
  191--222. Int. Press, Somerville, MA, 2016.

\bibitem[GGLR23]{GGLR}
M.~Green, P.~Griffiths, R.~Laza, and C.~Robles.
\newblock Period mappings and properties of the augmented {H}odge line bundle.
\newblock arXiv:1708.09523, 2023.

\bibitem[GGR25]{GGR-part1}
Mark Green, Phillip Griffiths, and Colleen Robles.
\newblock Analog of {S}atake--{B}aily--{B}orel for period maps.
\newblock arXiv:2010.06720, 2025.

\bibitem[Gri68]{MR0229641}
Phillip~A. Griffiths.
\newblock Periods of integrals on algebraic manifolds. {I}. {C}onstruction and
  properties of the modular varieties.
\newblock {\em Amer. J. Math.}, 90:568--626, 1968.

\bibitem[GS69]{MR0259958}
Phillip Griffiths and Wilfried Schmid.
\newblock Locally homogeneous complex manifolds.
\newblock {\em Acta Math.}, 123:253--302, 1969.

\bibitem[KP16]{MR3474815}
Matt Kerr and Gregory Pearlstein.
\newblock Boundary components of {M}umford-{T}ate domains.
\newblock {\em Duke Math. J.}, 165(4):661--721, 2016.

\bibitem[Rob17]{MR3701983}
Colleen Robles.
\newblock Degenerations of {H}odge structure.
\newblock In {\em Surveys on recent developments in algebraic geometry},
  volume~95 of {\em Proc. Sympos. Pure Math.}, pages 267--283. Amer. Math.
  Soc., Providence, RI, 2017.
\newblock arXiv:1607.00933.

\bibitem[Rob25]{Robles-pseudocnvx-eg}
Colleen Robles.
\newblock Pseudoconvexity at infinity in hodge theory: a codimension one example, in \emph{Current developments in Hodge theory: Proceedings of Hodge theory at IMSA}, pages 141--181.
\newblock Springer, Cham (2025), doi:10.1007/978-3-031-99683-2\_6.

\bibitem[Sch73]{MR0382272}
Wilfried Schmid.
\newblock Variation of {H}odge structure: the singularities of the period
  mapping.
\newblock {\em Invent. Math.}, 22:211--319, 1973.

\end{thebibliography}
\def\cprime{$'$} \def\Dbar{\leavevmode\lower.6ex\hbox to 0pt{\hskip-.23ex
  \accent"16\hss}D}


\end{document}